\documentclass[a4paper,11pt]{article}

\input{styleart.sty}
\usepackage{nicefrac}

\date{}
\begin{document}

\author{Ayala Dente-Byron \footnote{Einstein Institute for mathematics, Hebrew University of Jerusalem, and Mathematics Department, Technion, Israel Institute of Technology.} \and Chloe Perin\footnote{Einstein Institute for mathematics, Hebrew University of Jerusalem. This research was partially supported by the ISRAEL SCIENCE FOUNDATION (grant No. 1480/16)}} 
\title{Homogeneity of torsion-free hyperbolic groups}
\maketitle

\begin{abstract} We give a complete characterization of torsion-free hyperbolic groups which are homogeneous in the sense of first-order logic, in terms of the JSJ decompositions of their free factors.
\end{abstract}

\section{Introduction}

To answer Tarski's famous question on the elementary equivalence of free groups of different ranks, Sela introduced in the early 2000's a new, geometric approach to study the first-order theory of free groups, and of hyperbolic groups in general (see \cite{Sel1} and other papers in the series). Sela's geometric approach allowed new directions of research in the field, yielding, among other, results about elementary submodels of free and hyperbolic groups \cite{Sel6,Sel7,PerinElementary}, and the homogeneity of the free group \cite{PerinSklinosHomogeneity} \cite{OuldHoucineHomogeneity}. 

A countable group $G$ is said to be homogeneous if tuples of elements which satisfy the same first-order formula are in the same orbit under $\Aut(G)$.

More precisely, let $\bar{u}=\left(u_{1},...,u_{k}\right)\in G^{k}$ be a $k$-tuple of elements of $G$. The type of $\bar{u}$ in $G$, denoted $tp^{G}\left(\bar{u}\right)$, is the set of all first-order formulae $\varphi\left(\bar{x}\right)$ in $k$ free variables $\bar{x}=\left(x_{1},...,x_{k}\right)$ in the language of groups, such that $G$ satisfies 
$\varphi\left(\bar{u}\right)$:
 
  $$tp^{G}\left(\bar{u}\right)=\left\{ \varphi\left(\bar{x}\right)|G\models\varphi\left(\bar{u}\right)\right\} $$
 
$G$ is homogeneous if whenever $\bar{u},\bar{v}$ are tuples such that $tp^{G}\left(\bar{u}\right)=tp^{G}\left(\bar{v}\right)$ there is an automorphism $f:G\to G$ such that 
$f\left(\bar{u}\right)=\bar{v}$.

In this paper we consider torsion-free hyperbolic groups, and give a full characterization of those which are homogeneous, in terms of their $JSJ$-decomposition.

Let us give an idea of the obstructions there can be to homogeneity of a torsion free hyperbolic group on a particular example. 

Suppose $G$ is a freely indecomposable torsion free hyperbolic group, and that its (maximal cyclic) JSJ decomposition consists of a single surface vertex, with associated surface of genus $3$ with a single boundary component, and a single non surface type vertex with associated vertex group $G_0$. Denote by $\langle z \rangle$ the edge group, by $S$ the surface type vertex group. Note that $S$ admits a presentation of the form 
$$\langle z, x_1, y_1, x_2, y_2, x_3, y_3 \mid z = [x_1, y_1][x_2, y_2][x_3, y_3]\rangle$$

Here is a description of a case in which the group $G$ is not homogeneous.
\begin{ex} Suppose that in $G_0$, the element $z$ is a commutator $[g, h]$. There is a retraction $G \to G_1 = G_0 *\langle x_1,y_1\mid \; \rangle$ defined by $x_2, y_2 \mapsto y_1, x_1$ and $x_3, y_3 \mapsto g, h$. We think of this as the subsurface $S_0$ corresponding to the group $\langle x_2,y_2,x_3,y_3 \rangle$ "floor-retracting" in the JSJ decomposition (see Definition \ref{SubsurfaceFloorRetractsDef}).

By a theorem of Sela (see Theorem \ref{ElementarySubgroups}), the subgroup $G_0 *\langle x_1,y_1\mid \; \rangle$ of $G$ is elementarily embedded in $G$, i.e. an element of $G_1$ has the same type in $G_1$ and in $G$. Now it is possible to show that the orbit of the commutator $[x_1, y_1]$ under $\Aut(G_1)$ grows exponentially, while its orbit under $\Aut(G)$ grows polynomially (indeed $[x_1, y_1]$ represents a simple closed curve on the surface, hence up to finite index so does any element in its orbit under $\Aut(G)$ - the set of such elements grows polynomially). These two orbits are therefore distinct - if $h$ lies in the former but not in the latter, its type in $G_1$ is the same as the type of $[x_1, y_1]$, hence their types in $G$ are the same, yet no automorphism of $G$ takes one to the other.
\end{ex}

A second type of obstruction is demonstrated by the following example.
\begin{ex} Assume now that $z$ is a product of three commutators $[g_1, h_1][g_2, h_2][g_3, h_3]$ in $G_0$. There is a retraction $G \to G_0$ defined by $x_i,y_i \mapsto g_i,h_i$. Then, again by the work of Sela, we have that $G_0$ is an elementary submodel of $G$.

Suppose moreover that $G_0$ admits a non trivial splitting over a trivial or a maximal cyclic subgroup. Then it must be that $z$ is not elliptic in this splitting, and from this we can deduce that the orbit of $z$ under $\Aut(G_0)$ is not equal to its orbit under $\Aut(G)$. By the same reasoning as in the previous example one shows that $G$ is not homogeneous. 
\end{ex}

Finally we consider a third type of obstruction
\begin{ex} Assume again that $z$ is a product of three commutators so that $G_0$ is an elementary subgroup.

Suppose now that the maximal cyclic JSJ decomposition of $G_0$ is trivial, but that $\Aut(G_0)$ is non trivial (though note that it has to be finite), and that it contains an automorphism which sends $z$ to an element $z'$ which is not conjugate to $z$.  Then $z$ and $z'$ have the same type in $G_0$, so they also have the same type in $G$, yet by canonicity of the JSJ decomposition no automorphism of $G$ sends $z$ to $z'$.
\end{ex}

The main result of the paper, Theorem \ref{MainGeneral}, states that these are essentially the only obstructions for a group to be homogeneous.

\section{JSJ and modular groups}

In this section, we give an overview of the definitions and results we will need from Bass-Serre theory and the theory of JSJ decompositions. For further reference, the reader is invited to consult \cite{SerreTrees} and \cite{GuirardelLevittUltimateJSJ} respectively. We also prove several ad hoc results about JSJ decomposiitions that will be useful in the sequel.

\subsection{Graphs of groups} \label{GOGSec}

Given a group $G$ we call $G$-tree a tree endowed with an action of $G$ without inversion of edges. We sometimes work in the relative setting where we are given a group $G$ and a finite collection ${\cal H}=(H_1, \ldots, H_m)$ of subgroups of $G$, in that case we consider $(G, {\cal H})$-trees, that is, $G$-trees in which each $H_i$ is elliptic.

Bass-Serre theory associates to each such tree a graph of groups (for a complete reference, see \cite{SerreTrees}). To fix notation, we recall their definition: a graph of groups 
$$(\Lambda, \{G_v\}_{v \in V(\Lambda)},\{G_e\}_{e \in E(\Lambda)}, \{i_e\}_{e \in E(\Lambda)})$$
 is given by
\begin{enumerate}
\item an (oriented symmetric) connected underlying graph $\Lambda$ with endpoint maps $o: E(\Lambda) \to V(\Lambda)$ and $t: E(\Lambda) \to V(\Lambda)$ and edge involution $\bar{.}: E(\Lambda) \to E(\Lambda)$;
\item for each $v \in V(\Lambda)$, a vertex group $G_v$;
\item for each edge $e \in E(\Lambda)$ with $o(e)=v$, an edge group $G_e$ together with an embedding $i_e: G_e \to G_v$, with $G_{e} = G_{\bar{e}}$.
\end{enumerate}
We often abuse notation and write simply $\Lambda$ for the graph of groups above. We denote by $T_{\Lambda}$ the $G$-tree corresponding to $\Lambda$.

The fundamental group of such a graph of groups is defined as follows: let $\Lambda_0$ be a maximal subtree of $\Lambda$. Then the fundamental group $\pi_1(\Lambda, \Lambda_0)$ of $\Lambda$ relative to $\Lambda_0$ is generated by the groups $G_v$ together with a set $\{t_e\}_{e \in E(\Lambda)}$ of letters, subjected to the relations: 
\begin{itemize}
\item $t_et_{\bar{e}} =1 $ for each $e \in E(\Lambda)$;
\item $t_e =1$ for each $e \in \Lambda_0$;
\item $t_{e} i_{\bar{e}}(g) t_{\bar{e}} =i_e(g)$ for each $e \in E(\Lambda), g \in G_e$.
\end{itemize}
It can be shown that the isomorphism type of $\pi_1(\Lambda, \Lambda_0)$ does not depend on the choice of $\Lambda_0$.

The fundamental theorem of Bass-Serre theory is that if $T$ is a $G$-tree with associated graph of groups $\Lambda$, the fundamental group of $\Lambda$ is isomorphic to $G$.

\subsection{Surface type vertices}

Let $G$ be a group, and let  ${\cal H}=(H_1, \ldots, H_m)$ be a finite collection of subgroups of $G$. Let $\Lambda$ be a graph of groups associated to a $(G, {\cal H})$-tree.

\begin{defi} \label{SurfaceTypeVertex} (surface type vertex) A vertex $v$ in $\Lambda$ is called a surface type vertex if 
\begin{itemize}
\item there exists a compact connected surface $\Sigma$, with boundary, such that the vertex group $G_v$ is the fundamental group $S$ of $\Sigma$;
\item for each edge $e$ adjacent to $v$, the injection $i_e: G_e \hookrightarrow G_v$ maps $G_e$ onto a maximal boundary subgroup of $S$;
\item any maximal boundary subgroup of $S$ contains a conjugate either of $(i)$ the image $i_e(G_e)$ of exactly one of the edge groups adjacent to $v$, or $(ii)$ one of the subgroups $H_i$.
\end{itemize}
\end{defi}

\begin{defi} (graph of groups with surfaces) \label{GOGWithSurfaces} A graph of groups with surfaces is a graph of groups $\Lambda$ with a specified subset $V_{S}$ of vertices such that any vertex $v$ in $V_S$ is of surface type. By a slight abuse of language (since some vertices that are not in $V_S$ may satisfy the conditions of Definition \ref{SurfaceTypeVertex}) we call the vertices in $V_S$ the surface type vertices of $\Lambda$. 

A vertex $v$ of the tree $T_{\Lambda}$ corresponding to $\Lambda$ whose projection to $\Lambda$ is of surface type is also said to be of surface type. The surfaces corresponding to surface type vertices of $\Lambda$ are called the surfaces of $\Lambda$.
\end{defi}

We have the following useful property
\begin{rmk} \label{1Acylindricity} Let $T$ be a tree endowed with an action of a group $G$ corresponding to a graph of group with surfaces $\Lambda$, and let $v$ be a surface type vertex of $T$. Then the action is $1$-acylindrical at $v$, that is, if an element $g \in G$ stabilizes two distinct edges adjacent to $v$, it must be trivial. This is a straightforward consequence of the malnormality of maximal boundary subgroups in surface groups.
\end{rmk}

We need to slightly generalize this to the notion of a graph of groups with sockets. Essentially, a socket type vertex has vertex group a surface groups amalgamated with cyclic groups along its boundaries, adding roots to maximal boundary elements. We thus give
\begin{defi} (socket type vertices) \label{SocketVertexDef} A vertex $v$ in $\Lambda$ is called a socket type vertex if there exists a compact connected surface with boundary $\Sigma$, with fundamental group $S$ and $B_1, \ldots, B_k$ representatives of the conjugacy classes of maximal boundary subgroups of $S$ such that
\begin{itemize}
\item  the vertex group $G_v$ is the amalgamated product 
of $S$ of $\Sigma$ with cyclic groups $\hat{B_1} ,...,\hat{B_k}$ along  $B_1,...,B_k$ of $S$;
\item for each edge $e$ adjacent to $v$, the injection $i_e: G_e \hookrightarrow G_v$ maps $G_e$ onto one of the groups $\hat{B_i}$;
\item any one of the groups $\{\hat{B}_1,...,\hat{B}_k\}$ has a conjugate which contains either the image $i_e(G_e)$ of one of the edge groups adjacent to $v$, or one of the subgroups $H_i$.
\end{itemize}
\end{defi}

\begin{defi} (graph of groups with sockets) \label{GOGWithSockets} A graph of groups with sockets is a graph of groups $\Lambda$ with a specified subset $V_{S}$ of vertices such that any vertex $v$ in $V_S$ is a socket type vertex.

The vertices of $V_S$ as well as the corresponding vertices $v$ of the tree $T_{\Lambda}$ are called sockets. The surfaces of $\Lambda$ are the surfaces corresponding to its sockets.

The graph of groups with surfaces associated to $\Lambda$ is the graph of group obtained by refining $\Lambda$ at each socket type vertex by the graph of groups corresponding to the decomposition given in the first item of Definition \ref{SocketVertexDef} above. 
\end{defi}

\subsection{Modular group}

\begin{defi} A cyclic one-edge splitting of a group $G$ is a graph of groups $\Lambda$ with a single edge $e$ such that $\pi_1(\Lambda)\simeq G$ (note that the spanning tree is unique in this case) and $G_e$ is a cyclic group. 
\end{defi}

If $\Lambda$ is a cyclic one-edge splitting of $G$ then $G$ splits either as an amalgamated product: $G\simeq A\underset{C}{*}B$ or as an HNN-extention: $G\simeq A\underset{C }{*}$, where $C$ is cyclic. We define Dehn twists for both cases:

\begin{defi} Suppose $G\simeq A\underset{C}{*}B$ is an amalgamated product where $C$ is cyclic. The Dehn twist by an element $\gamma$ in the centralizer of $C$ in $A$ is the automorphism defined as follows:
\begin{itemize}
\item $\tau_{\gamma}(a)=a$ for every $a\in A$
\item$\tau_{\gamma}(b)=\gamma b \gamma^{-1}$ for every $b\in B$
\end{itemize}

Suppose $G\simeq A\underset{C}{*}$ is an HNN-extention of $A$  with $C = \langle c \rangle$ cyclic, that is, $G = \langle A, t \mid i_1(c)= ti_2(c)t^{-1} \rangle$ where $i_1, i_2$ are two embeddings of $C$ in $A$. The Dehn twist by an element $\gamma$ of the centralizer of $i_2(C)$ in $A$ is the automorphism defined as follows:
\begin{itemize}
\item $\tau_{\gamma}(a)=a$ for every $a\in A$
\item$\tau_{\gamma}(t)=t\gamma$. 
\end{itemize}
\end{defi}

\begin{defi} Let $G$ be a freely indecomposable group. The modular group of $G$, $Mod(G)$ is the subgroup of $\Aut(G)$ generated by Dehn twists associated to cyclic one edge splittings of $G$.
\end{defi}

\begin{rmk}
If $G$ is torsion-free hyperbolic, then in any one-edge splitting of $G$ over an infinite cyclic group, the edge group $C$ is maximal cyclic in one of the vertex groups. We will see in Section \ref{MaxCyclicJSJ} (paragraph on "Maximal splittings and JSJ decomposition") that $G$ also has a one edge splitting over the centralizer of $C$, which is a maximal cyclic (hence also maximal abelian) subgroup of $G$. Hence, $Mod(G)$ is generated by Dehn twists associated to maximal cyclic one edge splittings of $G$. 
\end{rmk}

The following is a classical result of Rips and Sela \cite[Corollary 4.4]{RipsSelaHypI}
\begin{thm} \label{ModFIInAut} Let $G$ be a torsion-free freely indecomposable hyperbolic group. Then $\Mod(G)$ has finite index in $\Aut(G)$.
\end{thm}

\subsection{JSJ decomposition} \label{JSJSec}

If $G$ is a torsion free hyperbolic group, a JSJ decomposition gives a way to encode all cyclic splittings of $G$, and hence to understand the modular group. Originally defined by Rips and Sela \cite{RipsSelaJSJ} it was further developed by Bowditch \cite{BowditchJSJ}, Fujiwara and Papasoglu \cite{FujiwaraPapasoglu}, Dunwoody and Sageev \cite{DunwoodySageev}. 

A more recent and encompassing reference on JSJ decomposition is the work of Guirardel and Levitt, which provides a general abstract framework for defining the JSJ decomposition of a group over a class of subgroups (see \cite{GuirardelLevittUltimateJSJ}). We give a summary of this framework, adapted to our purposes from the summary given in Section 4.1 of \cite{PerinSklinosForking}.

Let $G$ be a torsion free freely indecomposable finitely generated group. A {\em$G$-tree} is said to be {\em minimal} if it admits no proper $G$-invariant subtree.  
A {\em cyclic $G$-tree} is a $G$-tree whose edge stabilizers are infinite cyclic. 
If ${\cal H}= \{H_1, \ldots, H_m\}$ is a finite collection of subgroups of $G$, a {\em $(G,{\cal H})$-tree} is a $G$-tree in which each $H_i$ fixes a vertex. 
Following \cite{GuirardelLevittUltimateJSJ}, we call a (not necessarily simplicial) surjective equivariant 
map $d: T_1 \to T_2$ between two $(G,{\cal H})$-trees a {\em domination map}. We then say that $T_1$ dominates $T_2$. It is possible to show that $T_1$ dominates $T_2$ if any subgroup of $G$ which fixes a vertex in $T_1$ also fixes a vertex in $T_2$.
A surjective simplicial map $p:T_1 \to T_2$ which consists in collapsing some orbits of edges to points 
is called a {\em collapse map}. In this case, we also say that $T_1$ {\em refines} $T_2$.

\paragraph{Deformation space.} The {\em deformation space} of a cyclic $(G,{\cal H})$-tree $T$ is the set of all cyclic $(G,{\cal H})$-trees $T'$ such that $T$ dominates $T'$ and $T'$ dominates $T$. A cyclic $(G,{\cal H})$-tree is {\em universally elliptic} if its edge stabilizers are elliptic in every cyclic $(G,{\cal H})$-tree. If $T$ is a universally elliptic cyclic $(G,{\cal H})$-tree, and $T'$ is any cyclic $(G,{\cal H})$-tree, 
it is easy to see that there is a tree $\hat{T}$ which refines $T$ and dominates $T'$ (see \cite[Lemma 2.8]{GuirardelLevittUltimateJSJ}).

\paragraph{JSJ trees.} A cyclic relative {\em JSJ tree} for $G$ with respect to ${\cal H}$ is a universally elliptic cyclic $(G,{\cal H})$-tree which dominates 
any other universally elliptic cyclic $(G,{\cal H})$-tree. All these JSJ trees belong to a same deformation space, 
that we denote ${\cal D}_{JSJ}$. Guirardel and Levitt show that if $G$ is finitely presented and the $H_i$ are finitely generated, 
the JSJ deformation space always exists (see \cite[Corollary 2.21]{GuirardelLevittUltimateJSJ}). It is easily seen to be unique.

\paragraph{The tree of cylinders.} \label {TreeOfCylinders} In \cite[Section 7]{GuirardelLevittUltimateJSJ}, {\em cylinders} in cyclic $G$-trees are defined as equivalence 
classes of edges under the equivalence relation given by commensurability of stabilizers, 
and to any $G$-tree $T$ is associated its {\em tree of cylinders}. 
It can be obtained from $T$ as follows: the vertex set is the union 
$V_0(T_c) \cup V_1(T_c)$ where $V_0(T_c)$ contains a 
vertex $w'$ for each vertex $w$ of $T$ contained in at least two distinct cylinders, and $V_1(T_c)$ contains a vertex 
$v_c$ for each cylinder $c$ of $T$. There is an edge between vertices $w'$ and $v_c$ lying in $V_0(T_c)$ and $V_1(T_c)$ respectively 
if and only if $w$ belongs to the cylinder $c$.  

We get a tree which is bipartite: every edge in the tree of cylinders joins a vertex from $V_0(T_c)$
to a vertex of $V_1(T_c)$. Since the action of $G$ on $T$ sends cylinders to cylinders, the tree of cylinder admits an obvious $G$ action. 
Note also that if $H$ stabilizes an edge $e$ of $T$, its centralizer $C(H)$ preserves the cylinder containing $e$ since 
the translates of $e$ by elements of $C(H)$ are also stabilized by $H$: in particular there is a vertex in $T_c$ whose stabilizer is $C(H)$. 
It is moreover easy to see that this vertex is unique.

It turns out that the tree of cylinders is in fact an invariant of the deformation space 
\cite[Lemma 7.3]{GuirardelLevittUltimateJSJ}.

\paragraph{Case of freely indecomposable torsion-free hyperbolic groups.} By \cite[Theorem 9.18]{GuirardelLevittUltimateJSJ}, 
if $G$ is a torsion-free hyperbolic group freely indecomposable with respect to a finitely generated subgroup $A$, 
the tree of cylinders $T_c$ of the cyclic JSJ deformation space of 
$G$ with respect to $A$ is itself a cyclic JSJ tree, and it is moreover strongly $2$-acylindrical: namely, 
if a non-trivial element stabilizes two distinct edges, they are adjacent to a common cyclically stabilized vertex. 

Moreover, in this case the tree of cylinders is not only universally elliptic, but in fact universally compatible: namely, 
given any cyclic $(G,A)$-tree $T$, there is a refinement $\hat{T}$ of $T_c$ which \emph{collapses} onto 
$T$ \cite[Theorem 11.4]{GuirardelLevittUltimateJSJ}.

The JSJ deformation space being unique, it must be preserved under the action of $\Aut_A(G)$ on (isomorphism classes of) $(G,A)$-trees defined by twisting of the $G$-actions. 
Thus the tree of cylinders is a fixed point of this action, that is, 
for any automorphism $\phi \in \Aut_A(G)$, there is an automorphism 
$f:T_c \to T_c$ such that for any $x \in T_c$ and $g \in G$ we have $f(g \cdot x) = \phi(g) \cdot f(x)$.

\paragraph{Rigid and flexible vertices.} A vertex stabilizer in a (relative) JSJ tree is said to be {\em rigid} 
if it is elliptic in any cyclic $(G,{\cal H})$-tree, and {\em flexible} if not. In the case where $G$ is a torsion-free hyperbolic group and the $H_i$ are finitely generated subgroups of $G$ with respect to which $G$ is freely indecomposable, 
the flexible vertices of a the tree of cylinders of the cyclic JSJ deformation space of $G$ with respect to ${\cal H}$ are 
{\em surface type} vertices \cite[Theorem 6.6]{GuirardelLevittUltimateJSJ} in the sense of Definition \ref{SurfaceTypeVertex}. 
Moreover, the corresponding surfaces contain an essential simple closed geodesic so they cannot be thrice punctured spheres. 

\paragraph{Maximal splittings and JSJ-decomposition.} \label{MaxCyclicJSJ} Assume $G$ is a freely indecomposable and torsion free hyperbolic group.
We follow Section 9.5 in \cite{GuirardelLevittUltimateJSJ} to introduce a unique "maximal" tree in which edge stabilizers are maximal cyclic subgroups. This tree is obtained from the tree of cylinders $T_c$ of the cyclic JSJ deformation space in two steps. 

\begin{enumerate}
\item \textbf{Refine $T_c$ in special surface type vertices:} for every surface-type vertex $v$  with surface group a Klein bottle with one boundary component, $G_v$ is a free group $\left <s,t\right >$ with boundary element $tst^{-1}s$. We refine $T_c$ at $v$ by replacing $v$ with a graph with two vertices $v_1$  with vertex group generated by $\langle s, tst^{-1} \rangle$ and $v_2$ with vertex group generated by $s$, and two edges joining $v_1$ to $v_2$, with edge groups embedding in $G_{v_1}$ as $\langle s \rangle$ and $\langle tst^{-1}\rangle$ respectively and in $G_{v_2}$ as $\langle s \rangle$. 

\item \textbf{Replace the tree $T$ thus obtained by a tree in which all edge stabilizers are maximal cyclic in $G$:} for a cyclic subgroup $H$ of $G$, denote by $\hat {H}$ the maximal cyclic group containing it. Note that since $H$ is of finite index in $\hat{H}$, and $H$ is elliptic in $T$, then so is $\hat{H}$. Consider an edge $e$ such that $G_e$ is not maximal cyclic, so $e$ comes from an edge in $T_c$ and $\hat{G_e}$ fixes the adjacent vertex $v$ corresponding to a cylinder.

Identify $e$ with all the edges $g\cdot e$ for $g\in \hat {G_e}$. This is done simultaneously for all edges in the $G$-orbit of $e$. 
\end{enumerate}
Denote by $\hat{T}$ the tree obtained in this way from $T_c$. The tree $\hat{T}$  is a JSJ-tree for $G$ with respect to splittings of $G$ over maximal cyclic groups: it is universally elliptic with respect to such splittings and dominates every other universally elliptic maximal-cyclic $G$-tree (Proposition 9.30 in \cite{GuirardelLevittUltimateJSJ}). Since $T_c$ is canonical, so is $\hat{T}$. In particular:

\begin{rmk} \label{UniquenessMaxCyclicJSJ} 
If $\phi:H\to K$ is an isomorphism and $T$ is a $JSJ$-tree for $K$ then it is also a $JSJ$-tree for $H$, acting on $T$ via $\phi$. So if $T_K$ is the maximal cyclic $JSJ$ tree for $K$ obtained from a $JSJ$-tree $T$, its uniqueness implies the existence an isomorphism $f:T_H\to T_K$ respecting the $H$-action: for all $x\in T_H$ and $h\in H$, $f(h\cdot x)=\phi(h)\cdot f(x)$. 
\end{rmk} 

At the level of graphs of groups, the construction of the maximal cyclic decomposition $\hat{\Lambda}$ of $G$ from the JSJ decomposition $\Lambda$ corresponding to the tree of cylinders $T_c$ can be described as follows: for each vertex $v$ with non cyclic vertex group $H$ we denote by $\hat{H}$ the amalgamated product $H*_{E_1}\hat{E}_1* \ldots *_{E_l}\hat{E}_l$ where $E_1, \ldots, E_l$ are the images in $H$ of the edge groups of the edges $e_1, \ldots, e_l$ adjacent to $v$, and $\hat{E}_j$ is a copy of the (necessarily cyclic) vertex group corresponding to the other endpoint of $e_j$. 
	
Then $\hat{\Lambda}$ is obtained from $\Lambda$ by 1. refining surface type vertices corresponding to once punctured Klein bottles as described above to get a graph of groups $\tilde{\Lambda}$, and 2. replacing in $\tilde{\Lambda}$ each vertex group $H$ by $\hat{H}$ and each adjacent edge group $E_j$ by $\hat{E}_j$. Note that $\tilde{\Lambda}$ and $\Lambda'$ have the same underlying graph.

Note that flexible vertices of the maximal cyclic JSJ decomposition are socket type vertices.

\paragraph{Absolute vs. relative JSJ decomposition. } The following is a straightforward although useful remark.

\begin{rmk} \label{JSJVsRelJSJ} Let $G$ be a torsion-free hyperbolic group which is freely indecomposable with respect to a finite set of cyclic subgroups $(H_1, \ldots, H_k)$.
	
Suppose that the $H_i$ are elliptic in all the splittings of $G$ over trivial or maximal cyclic groups. Then $G$ is in fact freely indecomposable, since any free splitting $G = A* B$ of $G$ is a splitting relative to $(H_1, \ldots, H_r)$.

Moreover, the set of maximal cyclic $(G, {\cal H})$-trees is the same as the set of maximal cyclic $G$-trees. Hence the trees of cylinders are the same, and the maximal cyclic JSJ decomposition of $G$ is the same as
its maximal cyclic JSJ decomposition relative to $(H_1, \ldots, H_k)$.
\end{rmk}

\subsection{Modular group orbits and JSJ}

We will need the following result
\begin{lemma} \label{OrbitEdgeElement} Let $G$ be a freely indecomposable torsion free hyperbolic group, and denote by $\Lambda$ its maximal cyclic JSJ decomposition. If $g \in G$ stabilizes an edge of $T_{\Lambda}$, then the orbit of $g$ under $\Aut(G)$ consists of finitely many conjugacy classes.
\end{lemma}

\begin{proof} If $g$ stabilizes an edge of $T_{\Lambda}$ then it is elliptic in any one edge maximal cyclic splitting of $G$, hence it is sent to a conjugate by the Dehn twists associated to such a splitting. Thus its orbit under $\Mod(G)$ is contained it its conjugacy class. But $\Mod(G)$ has finite index in $\Aut(G)$ thus the orbit of $g$ under $\Aut(G)$ consists of finitely many conjugacy classes.
\end{proof}

We now want to show that the fundamental group of a subgraph of group of the JSJ admits precisely this subgraph as a JSJ decomposition relative to the adjacent edge groups. 

\begin{defi} (subgraph of groups) If $(\Lambda, \{G_v\}_{v \in V(\Lambda)},\{G_e\}_{e \in E(\Lambda)}, \{i_e\}_{e \in E(\Lambda)})$ is a graph of groups, and $\Lambda_1$ is a subgraph of $\Lambda$, we say that $(\Lambda_1, \{G_v\}_{v \in V(\Lambda_1)},\{G_e\}_{e \in E(\Lambda_1)}, \{i_e\}_{e \in E(\Lambda_1)})$ is a subgraph of groups of $\Lambda$ (i.e, we take the full vertex and edge groups of all the vertices and edges of the subgraph).
\end{defi}

Note that if $\Lambda_1$ is connected, then by general Bass-Serre theory, the embedding of $\Lambda_1$ in $\Lambda$ induces an embedding $\pi_1(\Lambda_1) \to \pi_1(\Lambda)$ which is defined up to conjugation by an element of $\pi_1(\Lambda)$.

We say that an edge $e$ of a graph of groups $\Lambda$ is {\em trivial} if one of its endpoints $v$ has valence $1$ in $\Lambda$ and $i_e(G_e)=G_v$.

\begin{prop} \label{RelativeJSJSubgraphOfGroups} Let $G$ be a freely indecomposable torsion free hyperbolic group, let $\Lambda$ be the decomposition coming from the action of $G$ on the tree of cylinders $T_{\Lambda}$ of the cyclic JSJ deformation space of $G$.
	
If $\Lambda_0$ is a connected subgraph of groups of $\Lambda$ with no trivial edges, and $G_0$ the corresponding subgroup of $G$, then $\Lambda_0$ is the JSJ decomposition coming from the tree of cylinders of the deformation space of $G_0$ relative to the collection $H^1_0, \ldots, H^r_0$ of edge groups of $\Lambda$ adjacent to $\Lambda_0$.
\end{prop}

The proof is based on Proposition 4.15 of \cite{GuirardelLevittUltimateJSJ}, which we quote here for convenience:
\begin{prop} \label{JSJFromUniversallyEll} Let $T$ be a universally elliptic $(G, {\cal H})$-tree. Then $G$ has a JSJ tree if and only if every vertex stabilizer $G_v$ of $T$ has a JSJ tree relative to the collection of incident edge groups. In this case, one can refine $T$ using these trees so as to get a JSJ tree of $G$. Conversely, if $v$ is a vertex of $T$ and $T_J$ is a JSJ decomposition for $G$ relative to ${\cal H}$, one can obtain a JSJ decomposition for $G_v$ relative to its collection of incident edge groups by considering the action of $G_v$ on its minimal subtree in $T_J$.
\end{prop}

\begin{proof}[Proof of Proposition \ref{RelativeJSJSubgraphOfGroups}] Consider the union $F(\Lambda_0)$ of all vertices and edges of $T_{\Lambda}$ mapped to $\Lambda_0$ by the quotient map, and pick a connected component $T'$. Then the action on $T'$ of its stabilizer is isomorphic to the action of the fundamental group $G_0$ of $\Lambda_0$ on the tree $T_{\Lambda_0}$ associated to $\Lambda_0$, so we identify them: $T' = T_{\Lambda_0}$.

Apply the lemma above to the tree $T$ obtained from $T_{\Lambda}$ by collapsing each component of $F(\Lambda_0)$ to a vertex - it is universally elliptic since $T_{\Lambda}$ is. Denote by $v_0$ the vertex $T_{\Lambda_0}$ is collapsed onto. The lemma tells us that the refinement of $T$ at $v_0$ by a JSJ tree $T^{JSJ}_{0}$ for $G_0$ relative to $H^1_0, \ldots, H^r_0$ is in the JSJ deformation space of $G$. Hence it is in the same deformation space as the refinement of $T$ at $v_0$ by  $T_{\Lambda_0}$ - that is, as $T_{\Lambda}$ itself, in other words, they have the same elliptic subgroups. In particular, the actions of $G_0$ on these two trees also have the same elliptic subgroups, hence they are in the same deformation space. Thus $T_{\Lambda_0}$ is also a JSJ tree for $G_0$ relative to $H^1_0, \ldots, H^r_0$.

To see that it is the tree of cylinders of the $JSJ$ deformation space, note that a tree is its own tree of cylinders if $(i)$ it is minimal and $(ii)$ its cylinders are stars and meet only in their non central vertex (a star is a graph which contains at least two edge, and all of whose edges are adjacent to a common central vertex). Note if two edges of $T_{\Lambda_0}$ are in the same orbit by $G$, they are in the same orbit under $G_0$ (any two translates of $T_{\Lambda_0}$ by elements of $G$ are disjoint or equal).

The minimal subtree $T_0$ of $G_0$ in $T_{\Lambda}$ is contained in $T_{\Lambda_0}$. Moreover, it is a union of orbits of edges under the action of $G_0$, which are the intersections of $T_{\Lambda_0}$ with orbits under the action of $G$. Thus $T_0$ corresponds to a subgraph of groups $\Lambda_1$ of $\Lambda_0$. Assume $\Lambda_1$ is a proper subgraph.
It is not hard to see that $\Lambda_0 - \Lambda_1$ is a union of trees, each connected by a unique edge $e$ to $\Lambda_1$ (otherwise some (product of) Bass-Serre elements, which is in $G_0$, does not preserve $T_{\Lambda_1}$). Now all the vertex and edge groups of this tree are contained in $G_e \simeq \Z$, but this contradicts the hypothesis that $\Lambda_0$ has no trivial edges. Thus $T_{\Lambda_0}$ is minimal.

A cylinder $C_0$ for the action of $G_0$ on $T_{\Lambda_0}$ is just the intersection of a cylinder $C$ in $T_{\Lambda}$ with $T_{\Lambda_0}$, since stabilizers in $G$ of edges of $T_{\Lambda_0}$ are in $G_0$. Suppose that $C_0$ consists of a single edge - this means that $C$ intersected $T_{\Lambda_0}$ in a single edge $e$, in particular no other edge of $C$ is in the orbit of $e$ under $G_0$, and thus under $G$. But this means that $G_e = G_v$ where $v$ is the center of $C$. This contradicts the hypothesis on $\Lambda_0$. Thus any cylinder in $T_{\Lambda_0}$ contains at least two edges.

Finally, if $C_0$ and $C'_0$ are two cylinders in the action of $G_0$ on $T_{\Lambda_0}$, coming from two cylinders $C$ and $C'$ in the action of $G$ on $T_{\Lambda}$, we have $C_0 \cap C'_0 \subseteq C \cap C'$, so $C_0$ and $C'_0$ intersect at most in one of their non central points. 
\end{proof}

We want to prove a similar result for the maximal cyclic JSJ decompositions.

\begin{defi} \label{GammaPlusDef} Let $G$ be a group, let $B_1, \ldots, B_q$ be a collection of maximal infinite cyclic subgroups of $G$. Suppose that the group $G^+$ is obtained from $G$ by amalgamating for each $i = 1, \ldots, r$ (for some $r \leq q$) an infinite cyclic group $B^+_i$ along $B_i$, where $B_i \leq B^+_i$ has finite index. 

Let $S$ be a $G$-tree with associated graph of groups $\Gamma$ in which each group $B_i$ for $i \leq r$ stabilizes a single vertex.

We denote by $\Gamma^+$ the graph of group obtained by adding for each $j \leq r$ a single edge attached to the vertex whose corresponding group contains (a conjugate of) $B_j$, with edge group $B_j$ and vertex group for the other endpoint $B^+_j$. We write  $S^+$ for the associated $G^+$-tree.
\end{defi}

\begin{rmk} Note that $(S^+)^{min}_G$ is isomorphic to $S$.
\end{rmk}

\begin{lemma} \label{LambdaPlusIsJSJ} Suppose $G$ is a freely indecomposable torsion free hyperbolic group, and let ${\cal B}= \{B_1, \ldots , B_q\}$ be a collection of cyclic subgroups of $G$. Denote by $\Lambda$ the tree of cylinder cyclic JSJ decomposition of $G$ relative to ${\cal B}$. Suppose that for any $j \leq r$ the subgroup $B_j$ does not stabilize an edge in any $(G, {\cal B})$-tree.

Let $B^+_1, \ldots , B^+_r$ be cyclic proper supergroups of finite index of $B_1, \ldots , B_r$. 

Then $\Lambda^+$ is the tree of cylinders JSJ decomposition of $G^+$ relative to 
$${\cal B}^+=\{B^+_1, \ldots, B^+_r, B_{r+1}, \ldots, B_q \}$$.
\end{lemma}

\begin{proof} An edge group of  $\Lambda^+$ is either an edge group of $\Lambda$, or one of the $B_i$ for $i \leq r$. Suppose $\Gamma$ is a cyclic splitting of $G^+$ in which all the subgroups of ${\cal B}^+$ are elliptic: it induces a splitting  of $G$ in which $B_1, \ldots, B_q$ are elliptic, and in such a splitting both edge groups of $\Lambda$ and the groups $B_i$ are elliptic. Thus the tree corresponding to $\Lambda^+$ is universally elliptic amongst cyclic $(G^+, {\cal B}^+)$-trees.	
	
We now want to show that the tree corresponding to $\Lambda^+$ dominates any other universally elliptic cyclic $(G^+, {\cal B}^+)$-tree $Y$. 

First, consider the minimal subtree $Y^{min}_G$ of $Y$ under the action of $G$, and show that it must be universally elliptic amongst $(G, {\cal B})$-trees: an edge stabilizer of $Y^{min}_G$ is of the form $H_G = H \cap G$ where $H$ is an edge stabilizer of $Y$. If $S$ is another $(G, {\cal B})$-tree, we can build the $(G^+,{\cal B}^+)$-tree $S^+$ as described in Definition \ref{GammaPlusDef}. Now note that $H$ is elliptic in $S^+$, hence $H_G = H \cap G$ must be elliptic in $S= (S^+)^{min}_G$.

Therefore, $T_{\Lambda}$ dominates $Y^{min}_G$, which is equivalent to saying that any subgroup of $G$ which is  elliptic in $T_{\Lambda}$ is elliptic in $Y^{min}_G$. Now if $K \leq G^+$ stabilizes a vertex $v$ in $T_{\Lambda^+}$, it is either a vertex stabilizer in $T_{\Lambda}$ or one of the $B^+_{j}$ - thus in both cases it fixes a vertex of $Y$. This proves that $T_{\Lambda}$ dominates $Y$. Finally, it is not hard to see that $T_{\Lambda^+}$ is its own tree of cylinders.
\end{proof}
 
We can now deduce that a subgraph of the maximal cyclic JSJ is the maximal cyclic JSJ of its fundamental group, relative to the adjacent edge groups:
\begin{prop} Let $\Lambda'$ be the maximal cyclic $JSJ$ decomposition for a freely indecomposable torsion free hyperbolic group $G$, and let $\Lambda'_0$ be a connected subgraph of groups of $\Lambda'$ with no trivial edges. Then $\Lambda'_0$ is the maximal cyclic JSJ decomposition of the subgroup $G'_0\leq G$ corresponding to $\Lambda'_0$, relative to the collection $H^1_0, \ldots, H^r_0$ of edge groups of $\Lambda'$ adjacent to $\Lambda'_0$.
\end{prop}

\begin{proof} Recall that if $\Lambda$ is the graph of groups corresponding to the tree of cylinders of the JSJ deformation space of $G$, for each vertex $v$ with non cyclic vertex group $H$ we denote by $\hat{H}$ the amalgamated product $H*_{E_1}\hat{E}_1* \ldots *_{E_l}\hat{E}_l$ where $E_1, \ldots, E_l$ are the images in $H$ of the edge groups of the edges $e_1, \ldots, e_l$ adjacent to $v$, and $\hat{E}_j$ is a copy of the (necessarily cyclic) vertex group corresponding to the other endpoint of $e_j$. 

Then $\Lambda'$ is obtained from $\Lambda$ by 1. refining surface type vertices corresponding to once punctured Klein bottles as described at the end of Section \ref{JSJSec} to get a graph of groups $\tilde{\Lambda}$ and 2. replacing in $\tilde{\Lambda}$ each vertex group $H$ by $\hat{H}$ and each adjacent edge group $E_j$ by $\hat{E}_j$ and deleting trivial edges. Recall that $\tilde{\Lambda}$ and $\Lambda'$ have the same underlying graph.

Note that (for example applying Proposition \ref{JSJFromUniversallyEll}) the graph of groups $\tilde{\Lambda}$ is the tree of cylinders JSJ decomposition of $G$ relative to the subgroups $Z_1, \ldots Z_m$ corresponding to the two-sided simple closed curves lying on once punctured Klein bottles that appear in $\Lambda$. 

Let $\tilde{\Lambda}_0$ be the subgraph of groups of $\tilde{\Lambda}$ whose underlying graph corresponds to $\Lambda'_0$. By Proposition \ref{RelativeJSJSubgraphOfGroups}, $\tilde{\Lambda}_0$ is the tree of cylinder cyclic JSJ decomposition of its fundamental group $\tilde{G}_0$ relative to the conjugates $Z'_1, \ldots Z'_n$ of the subgroups  which lie in $\tilde{G}_0$ and to the edge groups $B_1, \ldots, B_q$ of edges adjacent to $\tilde{\Lambda}_0$ which do not lie in $\tilde{\Lambda}_0$. Suppose that $r \leq q$ is such that $B_1, \ldots, B_r$ correspond to the edges adjacent to non cyclic type vertices of $\tilde{\Lambda}_0$ (so for $i \leq r$ the group $B_i$ lies in a non cyclic rigid type vertex or is a boundary subgroup of a surface type vertex group); while $B_{r+1}, \ldots, B_q$ are adjacent to cyclic type vertices, i.e. to centers of cylinders. 

Since $\tilde{\Lambda}_0$ corresponds to the the tree of cylinder cyclic JSJ decomposition of $\tilde{G}_0$ relative to $B_1, \ldots, B_q$ and $Z'_1, \ldots, Z'_n$, it has a common refinement with any $(\tilde{G}_0, (\{B_i\}_{i=1, \ldots,q}, \{Z'_j\}_{j=1, \ldots,n}))$-tree. We deduce that for $1, \ldots, r$ the group $B_i$ does not stabilize an edge in any such tree (since extra edges in the common refinement come from non boundary parallel simple closed curves on the surfaces, they cannot have $B_i$ as an edge group). 

Now $G'_0$ can be obtained from the fundamental group $\tilde{G}_0$ of $\tilde{\Lambda}_0$ by amalgamating supergroups of finite index $B^+_1, \ldots, B^+_r$ to the subgroups $B_1, \ldots, B_r$. 

The conditions of Lemma \ref{LambdaPlusIsJSJ} are satisfied, thus the tree of cylinders JSJ decomposition for $G'_0$ relative to $B^+_1, \ldots, B^+_r, B_{r+1}, \ldots, B_q, Z'_1, \ldots, Z'_n$ is precisely $\tilde{\Lambda}^+$ (i.e. it is obtained from $\tilde{\Lambda}_0$ as described in Definition \ref{GammaPlusDef}). 

It is not hard to see now that $\Lambda'_0$ is obtained from $\tilde{\Lambda}^+_0$ by replacing each vertex group $H$ by $\hat{H}$ and each adjacent edge group $E_j$ by $\hat{E}_j$. This means that $\Lambda'_0$ is precisely the maximal cyclic JSJ decomposition for $G'_0$ relative to $B^+_1, \ldots, B^+_r, B_{r+1}, \ldots, B_q, Z'_1, \ldots, Z'_n$, which are either the edge groups of $\Lambda'-\Lambda'_0$ adjacent to $\Lambda'_0$, or universally elliptic in maximal cyclic decompositions of $G_0$, and therefore of $G'_0$. 
\end{proof}

\section{Elementary submodels} \label{ElementarySubmodels}

Suppose a group $G$ admits a finitely generated elementary submodel $G'$. Then for $G$ to be homogeneous, it is necessary that any automorphism of $G'$ extend to an automorphism of $G$. Indeed, if $\theta' \in \Aut(G')$ and $u$ is a generating tuple for $G'$, the tuples $u$ and $\theta'(u)$ have the same type in $G'$, hence also in $G$. So if $G$ is homogeneous, there exists $\theta \in \Aut(G)$ which sends $u$ to $\theta'(u)$, in other words, which extends $\theta'$. It will thus be important to understand elementary submodels of torsion-free hyperbolic groups: this is the purpose of this section.

\begin{defi} \emph{(extended hyperbolic floor relative to a subgroup)} Let $G$ be a group, $H \leq G$ and let $r:G \to G'$ be a retraction. We say that $(G, G', r)$ is an \textbf{extended hyperbolic floor} relative to $H$ if there exists a graph of groups with surfaces $\Gamma$ such that
\begin{itemize}
\item $\Gamma$ is bipartite between surface and non surface type vertices;
\item each surface of $\Gamma$ is a punctured torus or has Euler characteristic at most $-2$;
\item there is a lift  $\Gamma^0$ of a maximal subtree of $\Gamma$ to the tree $T$ associated to $\Gamma$ such that if $H_1, \ldots, H_k$ are the stabilizers  of the non-surface type vertices of $\Gamma^0$, then $H\leq H_i$ for some $i$ and $G' = H_1 * \ldots * H_k$;
\item the image by $r$ of any surface type vertex group is non abelian; or $G'$ is cyclic and there exists a retraction $r': G * \Z \to G' * \Z$ which sends surface type vertex groups of $\Gamma$ to non abelian images.
\end{itemize}
If the first alternative holds in this last condition, we say that $(G,G',r)$ is a \textbf{hyperbolic floor}.
\end{defi}

\begin{defi} \emph{(extended hyperbolic tower)} \label{HypTower}
Let $G$ be a non cyclic group, let $H$ be a subgroup of $G$.
We say that $G$ is an \textbf{(extended) hyperbolic tower} over $H$ if there exists a finite
sequence $G=G^0 \geq G^1 \geq \ldots \geq G^m \geq H$ of subgroups of $G$ where $m \geq 0$ and:
\begin{itemize}
\item for each $k$ in $[0, m-1]$, there exists a  retraction $r_k:G^{k} \rightarrow G^{k+1}$
such that the triple $(G^k, G^{k+1}, r_k)$ is an (extended) hyperbolic floor relative to $H$;
\item $G^m = H * F * S_1 * \ldots * S_p$ where $F$ is a (possibly trivial) free group, $p \geq 0$, and each $S_i$
is the fundamental group of a closed surface without boundary of Euler characteristic at most $-2$.
\end{itemize}
If all the floors that appear are in fact (non extended) hyperbolic floors, we say $G$ is a \textbf{hyperbolic tower} over $H$.
\end{defi}

\begin{rmk} \label{FloorIsNotTower}
Note that if $(G, G', r)$ is an extended hyperbolic floor, it is not necessarily the case that $G$ is an extended hyperbolic tower over $G'$ - indeed, $G'$ might not be elliptic in the associated floor decomposition.

For example, if $G = \langle x,y, a,b \mid x^2y^2= a^2b^2 \rangle$, then $G$ has a floor structure over $G'=\langle x \rangle * \langle y \rangle$. The associated graph has one surface which is a twice punctured Klein bottle, with vertex group $\langle a, b, x^2, y^2 \mid x^2y^2= a^2b^2\rangle$, two non surface type vertices with associated group $\langle x \rangle$ and $\langle y \rangle$ and joined to the surface vertex by an edge with group $\langle x^2 \rangle$ and $\langle y^2 \rangle$ respectively. The retraction is given by $a \mapsto x, \; b \mapsto y$.
However it can be shown that $G$ is not a tower over $G'$, because there is no floor whose associated decomposition contains $\langle x, y \rangle$ in one of its non surface type vertices. 
\end{rmk}

This motivates the following definition
\begin{defi} \emph{(tower retract)} A subgroup $U \leq G$ is a tower retract of $G$ (respectively a tower retract relative to some subgroup $U' \leq U$) if there exists a sequence $G = G_0 > G_1 > \ldots > G_m=U$ with retractions $r_i: G_i \to G_{i+1}$ such that $(G_i, G_{i+1}, r_i)$ is a hyperbolic floor (in whose graph of groups decomposition $\Lambda_i$ the subgroup $U'$ is elliptic).
\end{defi}

Towers over a subgroup allow us to describe exactly finitely generated elementary extensions of torsion free hyperbolic groups. One implication of the following theorem is proved in \cite{PerinElementary}, and the other is due to \cite{SelaPrivate}
\begin{thm} \label{ElementarySubgroups} Let $H$ be a non cyclic subgroup of a torsion free hyperbolic group. Then $H$ is an elementary submodel of $\Gamma$ if and only if $\Gamma$ is a hyperbolic tower over $H$.
\end{thm}

\begin{rmk} \label{SubtowerIsNotAlwaysElementary} By Remark \ref{FloorIsNotTower}, the fact that $U$ is a tower retract of $G$ is not enough to show that $U$ is an elementary subgroup of $G$. In fact, in the example given in Remark \ref{FloorIsNotTower}, it is possible to show that $G'$ is not elementary in $G$.
\end{rmk}

However we can prove the following
\begin{prop} \label{TypesOfFactorOfSubtowerArePreserved}
Suppose that $(G, H, r)$ is a hyperbolic floor with associated decomposition for $H$ given by $H = H_0 * H_1 * \ldots * H_k$. Then the subgroup $\hat{H} = H_0* H_1^{a_1} * \ldots * H^{a_k}_k *\langle z \rangle$ is elementary in the group $\hat{G} = G * \langle a_1, \ldots, a_n, z \mid \; \rangle$.

Moreover for any tuple $\bar{g} \in H_0$, we have
$$ \tp^{H*\Z}(\bar{g}) = \tp^{G}(\bar{g})$$
If $H$ is non abelian, then in fact $\tp^{H}(\bar{g}) = \tp^{G}(\bar{g})$.
\end{prop}

\begin{proof} Let $\Lambda$ be the graph of groups associated to the hyperbolic floor $(G, H, r)$. We see the group $\hat{G}= G * \langle a_1, \ldots, a_k \mid \rangle*\Z$ as the fundamental group of the graph of groups obtained from $\Lambda$ by adding a trivially stabilized edge between the vertex $v_0$ corresponding to $H_0$ and each of the vertices $v_i$ corresponding to $H_i$ and a trivially stabilized loop from $v_0$ to itself, where the extra Bass-Serre elements are exactly $a_1, \ldots a_k$ and $z$ respectively. 

The fundamental group of the subgraph of groups consisting of exactly these $k+1$ new edges is precisely $ֿ\hat{H}=H_0* H_1^{a_1} * \ldots * H^{a_k}_k * \langle z \rangle$. There is an obvious embedding $j: H \to \hat{H}$ which is simply the identity on $H_0$ and conjugation by $a_i$ on $H_i$ for $i \geq 1$.

By collapsing these trivially stabilized edges we get a graph of groups $\hat{\Lambda}$ with a single non surface type vertex whose corresponding group is $\hat{H}$. We define a retraction $\hat{r}: \hat{G} \to \hat{H}$ by setting $\hat{r}\mid_G = j \circ r$, $\hat{r}(a_i) = 1$ and $\hat{r}(z)=z$. Now $(\hat{G}, \hat{H}, \hat{r})$ is a hyperbolic floor with associated decomposition $\hat{\Lambda}$. Since $\hat{H}$ is elliptic in $\hat{\Lambda}$, we have in fact that $\hat{G}$ is a tower over the non abelian subgroup $\hat{H}$, which by Theorem \ref{ElementarySubgroups} implies that $\hat{H}$ embeds elementarily in $\hat{G}$.

Now note that  by Theorem \ref{ElementarySubgroups}, we have that $G$ is elementarily embedded in $\hat{G}$, and   hence for any tuple $\bar{g} \in H_0$, we have
$$ \tp^{H * \Z}(\bar{g}) = \tp^{\hat{H}}(\bar{g}) =  \tp^{\hat{G}}(\bar{g}) = \tp^{G}(\bar{g}).$$

Now if $H$ is not abelian, a similar proof using only $k$ trivially stabilized edges from $v_0$ to each $v_i$ for $i>0$ shows that $G * \langle a_1, \ldots, a_k \mid \rangle$ is a hyperbolic tower over the subgroup $H_0* H_1^{a_1} * \ldots * H^{a_k}_k $, hence 
$$ \tp^{H}(\bar{g}) = \tp^{H_0* H_1^{a_1} * \ldots * H^{a_k}_k}(\bar{g}) =  \tp^{G * \langle a_1, \ldots, a_k \mid \rangle}(\bar{g}) = \tp^{G}(\bar{g})$$
for any tuple $\bar{g} \in H_0$.
\end{proof}

The following is a relative version of Definition 7.5 of \cite{Sel7}.
\begin{defi} \emph{(relative core)} \label{RelativeCoresDef} Let $\bar{u}$ be a tuple of elements in a torsion-free hyperbolic group $G$. The subgroup $U \leq G$ is a core of $G$ with respect to $\bar{u}$ if the following conditions are satisfied:
\begin{enumerate}
\item  there is a tower retract of $G$ with respect to the subgroup $\langle \bar{u} \rangle$ of the form $U * \F$ where $\F$ is a free group;
\item no factor of the Grushko decomposition of $U$ relative to $\langle \bar{u} \rangle$ is a free group except possibly the one containing $\langle \bar{u} \rangle$;
\item $U$ does not admit a retraction $r$ to a subgroup $r(U)$ such that $(U, r(U), r)$ is an extended hyperbolic floor with respect to $\langle \bar{u} \rangle$.
\end{enumerate}

\end{defi}
Such a relative core always exists. Indeed, a chain of successive floors $G_0 > G_1 > G_2 > \ldots$ is a chain of epimorphisms between subgroups of the torsion free hyperbolic group $G$, hence must stop by the descending chain condition for limit group, given by Theorem 1.12 of \cite{Sel7}.

\begin{rmk} \label{CoresHaveNoClosedSurfaces}Note that if $U_0 * U_1 * \ldots * U_s$ is the Grushko decomposition of a core $U$ of $G$ relative to $\bar{u}$, with $\bar{u} \in U_0$ then the maximal cyclic JSJ decomposition $\Lambda_i$ of $U_i$ (relative to $\bar{u}$ if $i=0$) cannot consist of a single surface type vertex. Indeed, the only fundamental group of a closed surface which is hyperbolic and does not admit a structure of extended hyperbolic floor over a proper subgroup is that of the connected sum of three projective planes, but by definition, the surface type vertices of a JSJ decomposition are the flexible vertices, which excludes this case.  
\end{rmk}

The following result will be central in the rest of the paper. Its proof is a refinement of the proof of Theorem 6.1 of \cite{PerinSklinosHomogeneity}.
\begin{thm} \label{EqualTypesImpliesIsomorphismOfCores} Let $G$ be a torsion free hyperbolic group. Let $\bar{u}, \bar{v}$ be two tuples in $G$. Let $U, V$ be cores of $G$ with respect to $\bar{u}, \bar{v}$ respectively.

If $\tp^G(\bar{u}) = \tp^G(\bar{v})$, there exists an isomorphism $U \to V$ sending $\bar{u}$ to $\bar{v}$.
\end{thm}

Recall from Section 7 of \cite{PerinSklinosHomogeneity} that if $\Lambda$ is a graph of groups with surfaces, two morphisms $h,h': \pi_1(\Lambda) \to G$ are said to be $\Lambda$-related if:
(i) they coincide up to conjugation on non surface type vertex groups of $\Lambda$, and 
(ii) for any surface type vertex group $S$, we have that $h(S)$ is non abelian if and only if $h'(S)$ is non abelian. In particular, if $\Lambda$ is a JSJ decomposition and $\sigma$ is a modular automorphism of a group $G$, then for any morphism $h: G \to G'$ the morphisms $h$ and $h \circ \sigma$ are $\Lambda$-related. 

We extend this definition as follows
\begin{defi} Let $h,h': H_1 *\ldots * H_m \to G$ be morphisms, and let $\Lambda_1, \ldots, \Lambda_m$ be graph of groups with surfaces such that $H_i = \pi_1(\Lambda_i)$. We say $h,h'$ are $(\Lambda_1, \ldots, \Lambda_m)$-related if $h\mid_{H_i}$ is $\Lambda_i$-related to $h'\mid_{H_i}$ for each $i$.
\end{defi}

We will need the following definition
\begin{defi} \label{FactorInjectiveDef} \emph{(factor injective morphisms)} Let $H$ and $G$ be hyperbolic groups, and let  $H = H_1 *\ldots * H_m$ be a (relative) Grushko decomposition for $H$. We say that a morphism $h: H \to G$ is \textbf{factor injective} if its restriction to each $H_i$ is injective, and $h(H_i)$ is not a subgroup of a conjugate of $h(H_j)$ for any $i \neq j$.
\end{defi}

\begin{lemma} \label{ComposingWithRetractionisRelated} Let $H$ and $G$ be hyperbolic groups. Suppose $H = H_0* H_1 *\ldots * H_m$ is the Grushko decomposition for $H$ relative to some tuple $\bar{u}$, where $\bar{u} \in H_0$, and suppose $\Lambda_i$ is the JSJ decomposition of $H_i$ corresponding to the tree of cylinder (relative to $\bar{u}$ if $i=0$). 

Let $h: H \to G$ be a morphism, and suppose $(G, G', r)$ is a hyperbolic floor structure for $G$ relative to $h(\bar{u})$.

If $h$ is injective on each $H_i$ then $r \circ h$ is $(\Lambda_0, \ldots, \Lambda_m)$-related to $h$. 
\end{lemma}

\begin{proof} Denote by $\Gamma$ the decomposition associated to the hyperbolic floor $(G,G',r)$. Since $h \mid_{U_i}$ is injective, $\Gamma$ induces a decomposition for $U_i$ via $h$.
	
If $R$ is a non surface type vertex of $\Lambda_i$, which is the JSJ decomposition of $U_i$ corresponding to the tree of cylinders, it must be elliptic in this induced decomposition and cannot correspond to a surface vertex group, hence $h(R)$ lies in a non surface type vertex group of $\Gamma$. This means that up to conjugation it lies in $G'$, in particular $r$ restricts to a conjugation on $h(R)$.

If $S$ is a surface vertex group of $\Lambda_i$, its boundary subgroups are universally elliptic in cyclic $U_i$-trees. In particular they are elliptic in the decomposition of $U_i$ induced by $\Gamma$. The decomposition $\Delta$ of $S$ induced by $\Gamma$ via $h$ is thus dual to a set of disjoint simple closed curves on $S$ (see Theorem III.2.6 \cite{MorganShalen}). If no vertex group of $\Delta$ is sent by $h$ to a surface type vertex group of $\Gamma$, by bipartism of $\Gamma$ and $1$-acylindricity near surface type vertices, we see that in fact $h(S)$ lies in a single non surface type vertex of $\Gamma$. In particular $r$ restricts to a conjugation on $h(S)$, thus $r(h(S))$ is abelian iff $h(S)$ was. If some vertex group $S_1$ of $\Delta$ is sent by $h$ to a surface type vertex group $\hat{S}$ of $\Gamma$, by Lemma 3.10 in \cite{PerinElementary} the group $h(S_1)$ in fact has finite index in $\hat{S}$ and in particular $h(S_1)$ (and thus $h(S)$) is non abelian. Since $r(\hat{S})$ is not abelian by definition of a hyperbolic floor, $r(h(S_1))$ is also non abelian ($r(\hat{S})$ is a non abelian subgroup of a torsion free hyperbolic group, so it cannot be virtually abelian). Thus both $h(S)$ and $r(h(S))$ are non abelian. 
\end{proof}

We also need the following result, which will help us express factor injectivity in first-order.
\begin{prop} \label{FactorInjOrKillSomeone} Let $U, G$ be torsion free hyperbolic groups, and let $\bar{u}$, $\bar{g}$ be tuples in $U,G$ respectively. Suppose that the Grushko decomposition of $U$ relative to $\bar{u}$ is of the form $U = U_0 * \ldots * U_m$ where $\bar{u} \in U_0$ and for each $i \geq 1$, the factor $U_i$ is neither a free group, nor a closed surface group. Let $\Lambda_i$ denote the JSJ decomposition of $U_i$ (relative to $\bar{u}$ if $i=0$).

There exist finitely many elements $w_1, \ldots , w_r$ in $U$ such that if $h: U \to G$ satisfies $h(\bar{u})=\bar{g}$ and is not factor injective, then there is  a morphism $h':U \to G$ which is $(\Lambda_1, \ldots, \Lambda_m)$-related to $h$ and satisfies $h'(w_k)=1$ for some $k$.

Moreover, if some morphism $f$ from $U$ to any group satisfies $f(w_k)=1$ for some $k$, then either $f$ is non injective on one of the factors $U_i$, or $f(U_i)$ intersects non trivially $f(U_j)$ for some $i \neq j$.
\end{prop}

\begin{proof}  For each factor $U_i$, there exists a finite set of elements $u^i_1, \ldots, u^i_{r_i}$ such that for any morphism $h_i: U_i \to G$ which is not injective (and sends $\bar{u}$ to $\bar{g}$ if $i=0$), there exists $h'_i$ which is $\Lambda_i$ related to $h_i$ and kills one of the $u^i_k$ (see Proposition 1.25 of \cite{Sel7}, and Theorem 4.4 in \cite{PerinSklinosHomogeneity} for the relative version). 
	
Thus if $h: U \to G$ is a morphism with $h(\bar{u})=\bar{g}$ which is not injective on some factor of $U$, there exists $h': U \to G$ which is $(\Lambda_1, \ldots, \Lambda_m)$-related to $h$ and kills one of the $u^i_k$.

Now by Theorem 4.1 of \cite{PerinSklinosHomogeneity} (which generalizes \cite[Theorem 7.1]{RipsSelaHypI}), for each pair $i \neq 0, j \neq 0$ with $i \neq j$, there are  finitely many embeddings $U_i \to U_j$ up to precomposition by an element of $\Mod(U_i)$ and postcomposition by a conjugation by an element of $U_j$. Pick a non trivial element $r_i$ in a non surface vertex group of $\Lambda_i$. Since modular automorphisms restrict to a conjugation on non surface vertex groups, the images of $r_i$ by embeddings $U_i \to U_j$ fall in finitely many conjugacy classes, we choose $t^i_1, \ldots t^i_{n_i}$ some representatives of these classes (over all values of $j$ which are distinct from $i$). If $h$ is injective on all the free factors of $U$, but sends $U_i$ into a conjugate of some $h(U_j)$, then some morphism $h'$ which coincides with $h$ up to conjugation on $U_j$ will satisfy $h'(r_i)=h'(t^i_k)$, thus $h'$ kills one of the non trivial elements $r^{-1}_i t^i_k$.

We thus take as $\{w_1, \ldots, w_r\}$ the union of the sets $\{u^i_j\}_{i,j}$ and $\{r^{-1}_i t^i_k\}_{i,k}$.

If some morphism $f: U \to \Gamma$ is injective on each of the factors $U_i$, and kills one of the $w_l$, it must kill one of the elements $r^{-1}_i t^i_k$. Then $f(r_i) = f(t^i_k)$, but since $r_i \in U_i$ and $f$ is injective on $U_i$, we have $f(r_i) \neq 1$. In that case $f(U_i)$ and $f(U_j)$ intersect non trivially, which proves the claim.
\end{proof}

The key step to prove Theorem \ref{EqualTypesImpliesIsomorphismOfCores} is
\begin{prop} \label{SameTypeImpliesFactorInjOnCores} Let $G$ be a torsion free hyperbolic group. Let $\bar{u}, \bar{v}$ be two tuples in $G$. Let $U, V$ be cores of $G$ with respect to $\bar{u}, \bar{v}$ respectively.

If $\tp^G(\bar{u}) = \tp^G(\bar{v})$, there exists a factor injective morphism $U \to V$ sending $\bar{u}$ to $\bar{v}$. 
\end{prop}

\begin{proof} To prove the claim, we proceed by contradiction: we will assume that $\tp^G(\bar{u}) = \tp^G(\bar{v})$, but that there is no factor injective morphism $U \to V$ sending $\bar{u}$ to $\bar{v}$, and get a contradiction to the fact that $U$ is a core.

Let $U = U_0 * \ldots * U_m$ be the Grushko decomposition of $U$ relative to $\bar{u}$, where $\bar{u} \in U_0$ and for each $i \geq 1$, let $\Lambda_i$ denote the JSJ decomposition of $U_i$ (relative to $\bar{u}$ if $i=0$).	
	
\paragraph{Step 1: Show that for any morphism $h: U \to G$ sending $\bar{u}$ to $\bar{v}$, there exists $h: U \to G$ which is $(\Lambda_0, \ldots, \Lambda_m)$-related to $h$ and is not factor injective.}	
	
Let $h: U \to G$ be a morphism sending $\bar{u}$ to $\bar{v}$. If $h$ itself is not factor injective, we are done. If $h$ is factor injective, and if  $(G, G^1, r_1)$ denotes the first floor of the structure of hyperbolic tower that $G$ admits over $V$, then by Lemma \ref{ComposingWithRetractionisRelated}, the composition $r_1 \circ h$ is $(\Lambda_0, \ldots, \Lambda_m)$-related to $h$. If $r_1 \circ h$ itself is factor injective, we keep going. If not, we set $h' = h \circ r_1$. By induction, and by our hypothesis that there is no factor injective morphism $U \to V$ sending $\bar{u}$ to $\bar{v}$ we see that eventually we get a morphism $h'$ which is $(\Lambda_0, \ldots, \Lambda_m)$-related to $h$ and is not factor injective.   

\paragraph{Step 2: Express this as a first-order formula $\phi(\bar{x})$ that $\bar{v}$ satisfies.}

By Proposition \ref{FactorInjOrKillSomeone}, there are elements $w_1, \ldots, w_r$ of $U$ such that for any morphism $h: U \to G$ sending $\bar{u}$ to $\bar{v}$ there is a $(\Lambda_0, \ldots, \Lambda_m)$-related morphism $h': U \to G$ which kills one of the elements $w_1, \ldots, w_r$.

We can write down a first-order sentence $\phi(\bar{v})$ whose interpretation on $G$ expresses precisely this. 

\paragraph{Step 3: Find a morphism $U \to G$ with $\bar{u} \mapsto \bar{u}$ which is $(\Lambda_0, \ldots, \Lambda_m)$-related to the identity and non injective on one of the factors $U_i$.}

Since $\tp^G(\bar{u}) = \tp^G(\bar{v})$, we must have $G \models \phi(\bar{u})$, that is, for any morphism $h: U \to G$ fixing $\bar{u}$ there is a $(\Lambda_0, \ldots, \Lambda_m)$-related morphism which kills one of the elements $w_1, \ldots, w_r$. We apply this to the identity embedding of $U$ in $G$ to get a morphism $f: U \to G$ which is $(\Lambda_0, \ldots, \Lambda_m)$-related to the identity and kills one of the $w_i$. We now successively compose $f$ with the floor retractions between $G$ and $U$ and prove that at some point we must get a morphism with the required properties.

If $f$ is non injective on one of the factors $U_i$, we are done. Suppose thus that $f$ is injective on each of the $U_i$. Let $r_1$ be the retraction of the first floor structure of $G$ over $U$: by Lemma \ref{ComposingWithRetractionisRelated}, the morphism $r_1 \circ f$ is still $(\Lambda_0, \ldots, \Lambda_m)$-related to the identity, and kills one of the elements $w_i$. If $(r_1 \circ f)\mid_{U_i}$ is non injective for some $i$, we are done. If not, we repeat the process by composing it with the next floor retraction.

If the process doesn't stop until the last floor structure, we get a morphism $f': U \to U$ which is $(\Lambda_0, \ldots, \Lambda_m)$-related to the identity, and kills one of the elements $w_i$. If $f'$ is injective on each of the $U_i$, by the moreover part of Proposition \ref{FactorInjOrKillSomeone} there exist distinct indices $i,j$ such that some conjugates of $f(U_i)$ and $f(U_j)$ intersect. But let $R_i$ be a non surface type vertex of $\Lambda_i$ (there is such a vertex by Remark \ref{CoresHaveNoClosedSurfaces}): since $f'$ is $(\Lambda_0, \ldots, \Lambda_m)$-related to the identity, we have that $f(R_i)$ is a conjugate of $R_i$, in particular it is contained in a conjugate of $U_i$. Since $f'(U_i) \simeq U_i$ is freely indecomposable, $f'(U_i)$ is in fact contained in some conjugate of $U_i$. Similarly, $f(U_j)$ lies in a conjugate of $U_j$. Thus the restriction of $f$ to one of the factors $U_1, \ldots, U_m$ must be non injective.

\paragraph{Step 4: Deduce that $U$ is not a core.}
Therefore, for some $i$, the  morphism $f \mid_{U_i}$ is a non injective preretraction $r: U_i \to G$ relative to $\Lambda_i$ (see Remark 6.5 of \cite{PerinSklinosHomogeneity}). Now by applying Theorem 5.10 of \cite{PerinSklinosHomogeneity} repeatedly, we get a non injective preretraction $U_i \to U_i$ relative to $\Lambda_i$. Finally, by applying Theorem 5.9 of \cite{PerinSklinosHomogeneity}, we see that $U_i$ must admit a structure of hyperbolic floor (in whose corresponding floor decomposition  $\bar{u}$ is elliptic if $i=0)$. This structure can be extended to a structure of hyperbolic floor for $U$ relative to $\bar{u}$, which contradicts the fact that $U$ is a core.
\end{proof}

We can now complete the proof of Theorem \ref{EqualTypesImpliesIsomorphismOfCores}.
\begin{proof}Let $U = U_0 * \ldots * U_m$ be the Grushko decomposition of $U$ relative to $\bar{u}$, where $\bar{u} \in U_0$ and let $V = V_0 * \ldots * V_n$ be the Grushko decomposition of $V$ relative to $\bar{v}$, where $\bar{v} \in V_0$.

By Proposition \ref{SameTypeImpliesFactorInjOnCores}, there exists a factor injective morphism $f:U \to V$ sending $\bar{u}$ to $\bar{v}$. By symmetry, there exists a factor injective morphism $h:V \to U$ sending $\bar{v}$ to $\bar{u}$. 

Since $f$ is injective on the factors $U_i$, the image $f(U_i)$ is freely indecomposable and non cyclic, hence it must be contained in a conjugate of one of the factors $V_j$. Similarly for any $j$ we have that $h(V_j)$ is contained in one of $U_1, \ldots, U_m$.

Now there are finitely many factors so for some integer $k$, the morphism $(h \circ f)^k$ embeds one of the factors $U_i$ in a conjugate of itself, and fixes $\bar{u}$ if $i=0$. By (relative) co-Hopf property, it must be an isomorphism, hence all the factors involved in the successive images of $U_i$ under the homomorphisms $(h\circ f)^s$ for $s\leq k$ are isomorphic, and up to reordering and composing $f$ by a conjugation on each factor, we may assume that $f$ sends $U_0 * U_1 * \ldots * U_l$ isomorphically onto $V_0*\ldots *V_l$ for some $l \leq m$. Now the image by $f$ of any other factor $U_j$ for $j>l$ cannot be sent into a conjugate of a factor $V_k$ for $k<l$, since $V_k =f(U_i)$ for some $i \leq l$, and this would contradict the fact that $f$ is factor injective.  We can thus repeat the argument for the restriction of $f$ to $U_{l+1}* \ldots * U_m$ and the restriction of $h$ to $V_{l+1}* \ldots *V_n$.  
\end{proof}

We get the following corollary
\begin{cor} \label{HomogeneityIsExtendingAut} Let $G$ be a torsion free hyperbolic group.
If any isomorphism $f: U \to V$ between two tower retracts of $G$ extends to an automorphism of $G$, then $G$ is homogeneous.
\end{cor}

\begin{proof} Let $\bar{u}, \bar{v}$ be tuples satisfying the same type in $G$, and let $U, V$ be cores for $G$ relative to $\bar{u}, \bar{v}$ respectively. By Theorem \ref{EqualTypesImpliesIsomorphismOfCores}, there is an isomorphism $U \to V$ sending $\bar{u}$ to $\bar{v}$, by assumption it extends to an automorphism $G \to G$ and thus we are done.
\end{proof}

If all tower retracts were elementarily embedded, the converse would hold, but we know this is not the case in general. However for us the following criterion for non homogeneity will be sufficient.

\begin{prop} \label{CriterionForNonHomogeneity} Let $(G, U, r_U)$ and $(G, V, r_V)$ be two hyperbolic floor structures.  Let $U=U_1* \ldots *U_l * F_U$ and $V = V_1* \ldots * V_m*F_V$ be the Grushko decompositions of $U$ and $V$ respectively.
	
Suppose that for each $i$ the factors $U_i$ and $V_i$ are isomorphic, and suppose there exists an isomorphism  $g$ either between $U_1$ and $V_1$, or between a free factor of $F_U$ and a free factor of $F_V$, which  fails to extend to an automorphism of $G$. Then $G$ is not homogeneous.
\end{prop}

\begin{proof}Assume $g$ is an isomorphism between $U_1$ and $V_1$, the proof in the second case is similar. Let $\bar{u}$ be a generating set for $U_1$. Let $U'=U_1* \ldots * U_{s}$ and $V'=V_1* \ldots * V_{s}$. 

Then $\tp^{U'}(\bar{u}) = \tp^{V'}(g(\bar{u}))$ (since there is an automorphism $U' \to V'$ sending $\bar{u}$ to $g(\bar{u})$). Now $U$ is the free product of $U'$ with a free group so $\tp^{U'}(\bar{u}) = \tp^{U*\Z}(\bar{u})$. Similarly $\tp^{V'}(g(\bar{u})) = \tp^{V*\Z}(g(\bar{u}))$.

By Proposition \ref{TypesOfFactorOfSubtowerArePreserved} we have in fact $\tp^{U*\Z}(\bar{u}) = \tp^G(\bar{u})$ and  $\tp^{V*\Z}(g(\bar{u})) = \tp^G(g(\bar{u}))$, so $\bar{u}$ and $g(\bar{u})$ have the same type in $G$. However, there is no automorphism of $G$ sending $\bar{u}$ to $g(\bar{u})$.
\end{proof}

Our strategy is now to identify from the JSJ of $G$ whether such floor structures exist.

\section{First obstruction to homogeneity: floor-retracting subsurfaces} \label{RetractingSubsurfaces}
\label{RetractingSubsurfacesSection}

The first obstruction to homogeneity of a group $G$ is the existence of a proper subsurface in one of the surfaces of the JSJ decomposition of $G$ which also appears in a decomposition associated to a floor structure of $G$. 

\subsection{Floor-retracting surfaces in a graph of groups - definition}

We first define a notion of floor retraction for (sets of) surfaces of a graph of groups. 

\begin{defi} \label{RetractingSurface} (floor-retracting set of surfaces) Let $\Lambda$ be a graph of groups with surfaces with fundamental group $G$. Let $W=\{v_1, \ldots, v_r\}$ be a set of surface type vertices with associated surfaces $\Sigma_1, \ldots, \Sigma_r$. Denote by $\Lambda_{W}$ the graph of groups with surfaces obtained by collapsing all edges not adjacent to one of the $v_i$.

If there exists an extended hyperbolic floor $(G, H,r)$ with associated graph of groups $\Lambda_{W}$, we say that \textbf{the set of surfaces $\Sigma_1, \ldots, \Sigma_r$ floor-retracts} in $\Lambda$.

If the set of surfaces $\{\Sigma \}$ floor-retracts in $\Lambda$ we simply say that $\Sigma$ floor-retracts in $\Lambda$.
\end{defi}

\begin{rmk} \label{HypFloorSurfacesRetract} Suppose $(G,G',r)$ is an extended hyperbolic floor with associated decomposition $\Lambda$. Any set of surfaces of $\Lambda$ floor-retracts in $\Lambda$. 
\end{rmk}

Let us have a closer look at what happens when a set of surfaces floor-retracts in a graph of groups.

\begin{defi} Suppose $\Sigma_1, \ldots, \Sigma_r$ is a set of surfaces which floor-retracts in a graph of groups $\Lambda$ with fundamental group $G$. By definition, there exists a retraction $r: G \to H = H_1 * \ldots * H_p$ (respectively $G*\Z \to H_1*\Z$ if $p=1$ and $H_1$ is cyclic). The restriction of $r$ to each $S_j=\pi_1(\Sigma_j)$ induces an action of the surface group $\pi_1(\Sigma_j$) on the tree corresponding to the free splitting $H_1 * \ldots * H_p$ in which the boundary subgroups are elliptic. This gives (for example by Lemma 3.4 in \cite{PerinElementary}) a collection ${\cal C}$ of simple closed curves on the surfaces $\Sigma_j$ whose corresponding elements are killed by $r$ - we may assume this collection is in fact maximal for this property. Thus $r$ factors as $r' \circ \rho_{\cal C}$ where $\rho_{\cal C}$ is the "pinching map" whose kernel is normally generated by the elements corresponding to ${\cal C}$.

Each group $\rho_{\cal C}(\pi_1(\Sigma_j))$ admits a graph of groups decomposition $\Gamma(\Sigma_j,{\cal C})$ with trivial edge stabilizers, and vertex stabilizers which are the fundamental groups of the connected components $\Sigma^1_j, \ldots , \Sigma^{m_j}_j$ obtained by cutting $\Sigma$ along the curves of ${\cal C}$ and gluing disks on the boundary thus created - see Figure \ref{PinchingFig}.
\end{defi}

\begin{figure}[ht!] \label{PinchingFig}
	\centering
	\includegraphics[width=.9\textwidth]{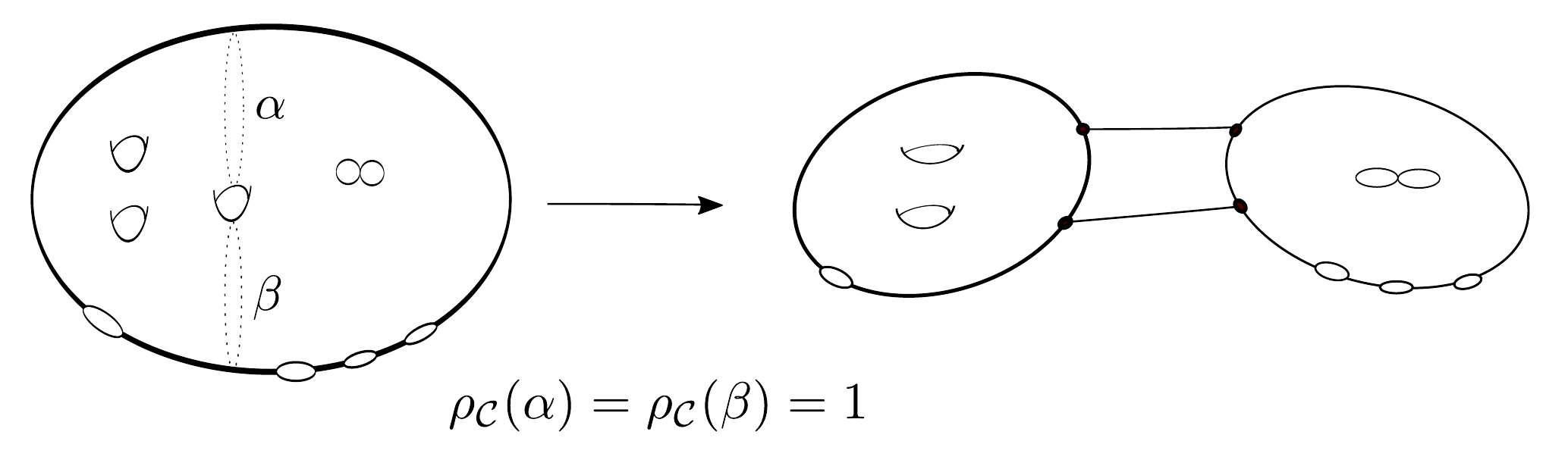}
	\caption{The construction of the graph of groups $\Gamma(\Sigma,{\cal C})$}
\end{figure}

Note that by construction each vertex group of $\Gamma(\Sigma_j,{\cal C})$ has image by $r'$ which lies in one of the factors $H_i$, and that $r'$ does not kill any element corresponding to a simple closed curve on one of the surface $\Sigma^k_j$.

\begin{rmk} In general, it is not true that if each one of the surfaces $\Sigma_1, \ldots, \Sigma_r$ floor-retracts in $\Lambda$ then the set $\{ \Sigma_1, \ldots, \Sigma_r \}$ floor-retracts: for example, if $\Lambda$ consists of one non surface type vertex $w$ with vertex group $G_w$, and two distinct surface type vertices with associated surfaces $\Sigma, \Sigma'$ which are both a once punctured torus whose boundary subgroups are identified to the same infinite cyclic subgroup $\langle z \rangle$ of $G_w$, then clearly both $\Sigma$ and $\Sigma'$ retract in $\Lambda$. However if $z$ is not a commutator in $G_w$, the set $\{\Sigma, \Sigma'\}$ does not retract in $\Lambda$.
\end{rmk}

But it is possible to show that this is essentially the only obstruction. Since we won't use this in the sequel we leave this to the interested reader to figure out. We will need however the following weaker version:

\begin{lemma} \label{SurfacesCanRetractTogether} Suppose that $\{\Sigma_1, \ldots, \Sigma_p\}$ is a floor-retracting set of surfaces in a graph of groups with surfaces $\Lambda$, and denote by $\{v_1, \ldots, v_p\}$ the corresponding surface type vertices. 

Suppose $\Sigma$ is a surface of $\Lambda$ whose corresponding vertex lies in a connected component $\Lambda_0$ of $\Lambda - \{v_1, \ldots, v_p\}$, and which floor-retracts in $\Lambda_0$. 

Then $\{\Sigma_1, \ldots, \Sigma_p, \Sigma \}$ floor-retracts in $\Lambda$.
\end{lemma}

\begin{proof}  Let $r: G \to G'$ be the retraction corresponding to $\{\Sigma_1, \ldots, \Sigma_p\}$, and  let $r_0: G'_0 \to G''_0$ be that corresponding to $\Sigma$ (where $G'_0=\pi_1(\Lambda_0)$). Note that by definition of a floor structure $G'$ admits a free decomposition as $G'_0* G'_1 * \ldots * G'_m$, where each $G'_i$ is the fundamental group of a connected component $\Lambda_i$ of $\Lambda- \{v_1, \ldots, v_p\}$. Let $\hat{r}: G' \to G''=G''_0 * G'_1* \ldots * G'_m$ be the retraction which restricts to $r'_0$ on $G'_0$ and to the identity on $G'_i$ for $i>0$.  
	
We claim that there is a retraction $r''$ such that $( G, G'', r'')$ is a hyperbolic floor with corresponding decomposition the graph of groups obtained from $\Lambda$ by collapsing all the edges which are not adjacent to one of the vertices $\{v, v_1, \ldots, v_p\}$. 

We first set $r'' = \hat{r} \circ r$ - though in some cases we will have to slightly modify this.	The only part that requires some care is to guarantee the non abelianity of images of surface groups. 
The image by $\hat{r} \circ r$ of the vertex group $S$ corresponding to $\Sigma$ is just $\hat{r}(S)$ which is non abelian. We now consider the vertex group $S_j$ associated to one of the vertices $v_j$. If $v_j$ is adjacent to at least two distinct connected components of $\Lambda - \{v_1, \ldots, v_p\}$ then $S_j$ contains non trivial elements in at least two of the factors $G'_j$, moreover if one of those is in $G'_0$ it fixes a vertex of $\Lambda'_0$ other than  $v$ so in particular it is sent to a conjugate of itself by $\hat{r}$. The image of $S_j$ by $\hat{r} \circ r$ is therefore non abelian. 

We therefore assume that $\Sigma_j$ is adjacent to only one connected component  of $\Lambda - \{v_1, \ldots, v_p\}$, say the component  $\Lambda'_i$ corresponding to $H_i$.
We choose a maximal set of simple closed curves on $\Sigma_1, \ldots, \Sigma_p$ whose corresponding elements are killed by $r$ and we build the morphism $\rho_{\cal C}$, and the graph of groups $\Gamma(\Sigma_j,{\cal C})$ as described above - recall that $r$ factors as $r' \circ \rho_{\cal C}$. 

If there is a vertex group $K$ of $\Gamma(\Sigma_j,{\cal C})$ which does not contain the image of boundary subgroups of $S_j$ (i.e. it corresponds to a closed surface $\Sigma^l_j$ of $\Gamma(\Sigma_j,{\cal C})$), and which admits a surjection onto $\Z$, then we can modify $r'$ by sending $K$ onto a cyclic subgroup of $G''$ which is not commensurable to the image in $G''_0$ of one of the boundary subgroups of $S_j$ say. This guarantees non abelianity of the image of $S_j$.

Suppose thus all the vertex group of $\Gamma(\Sigma_j,{\cal C})$ contain the image of some boundary subgroups of $S_j$. Consider the image of one such subgroup $K$ by $r'$ : it must lie in one of the factor $G'_i$.  Since the boundary subgroups are elliptic in the graph of groups decomposition $\Lambda_i$, this induces a decomposition $\Lambda^K_i$ of $K$ dual to a set of simple closed curves on the surface $\Sigma^l_j$ of $\Gamma(\Sigma_j,{\cal C})$ associated to $K$. In the case $i=0$, if one of the vertex groups of $\Lambda^K_i$ is sent to a conjugate of the vertex group $S$ associated to $v$, then its image has in fact finite index in $S$, since $r'$ does not kill any simple closed curve on $\Sigma^l_j$.  Thus, the image of $K$ by $\hat{r} \circ r$ contains a finite index subgroup of $\hat{r} \circ r(S) = \hat{r}(S)$ - in particular it is non abelian and we are done. If no vertex group of $\Lambda^K_i$ is sent to a conjugate of $S$, then $\hat{r} \circ r' \mid_{K}$ coincides with $r'$ up to a conjugation. If this is the case for all the vertices of $\Gamma(\Sigma_j,{\cal C})$, then it is not hard to see that the image of $S_j$ by $\hat{r} \circ r$ must be non abelian, unless all vertex groups of $\Gamma(\Sigma_j,{\cal C})$ are sent to conjugates of a same cyclic subgroup $Z$ by elements which lie in the kernel of $\hat{r}$. In that case, we can modify $r'$ by choosing the conjugating elements outside the kernel, and outside the cyclic subgroup $Z$ to guarantee that the image of $S_j$ by $\hat{r} \circ r$ is not abelian.
\end{proof}

The following result gives a combinatorial condition for a surface to floor-retract.
\begin{prop} \label{RetractingSurfaceEquations} Let $\Lambda$ be a graph of groups with surfaces, and let $S$ be a surface type vertex group of $\Lambda$ with associated surface $\Sigma$. Assume $\Sigma$ is an $n$-punctured connected sum of $l$ tori and $m$ projective planes.
	
Denote by $H_1, \ldots, H_p$ the vertex groups of $\Lambda_{\Sigma}$ other than $S$. For each $i$, denote by $e^i_1, \ldots, e^i_{n_i}$ the edges of $\Lambda_{\Sigma}$ adjacent to the vertex corresponding to $H_i$. 
 
Then $\Sigma$ floor-retracts in $\Lambda$ if and only if there exist $l_1, \ldots, l_p, m_1, \ldots, m_p \in \N$ such that
\begin{enumerate}
\item for each $i$, there exist elements $\delta_{e^i_1}, \ldots, \delta_{e^i_{n_i}} $ of $ H_i$ (respectively of $H_1 *\Z$ if $p=1$ and $H_1$ is cyclic) which are conjugates of the generators of the edge groups of $e^i_j$ such that the product $\delta_{e^i_1} \ldots \delta_{e^i_{n_i}}$ can be written in $H_i$ as a product of $l_i$ non trivial commutators and $m_i$ non trivial squares;
\item if the elements $\delta_{e^i_j}$ and the squares appearing in all the above products all commute, then $l>0$;
\item  $l \geq \sum^{p}_{i=1} l_i$, and $m \geq \sum^p_{j=1} m_j$.
\end{enumerate}
\end{prop}

 \begin{figure}[ht!]
\centering
\includegraphics[width=.7\textwidth]{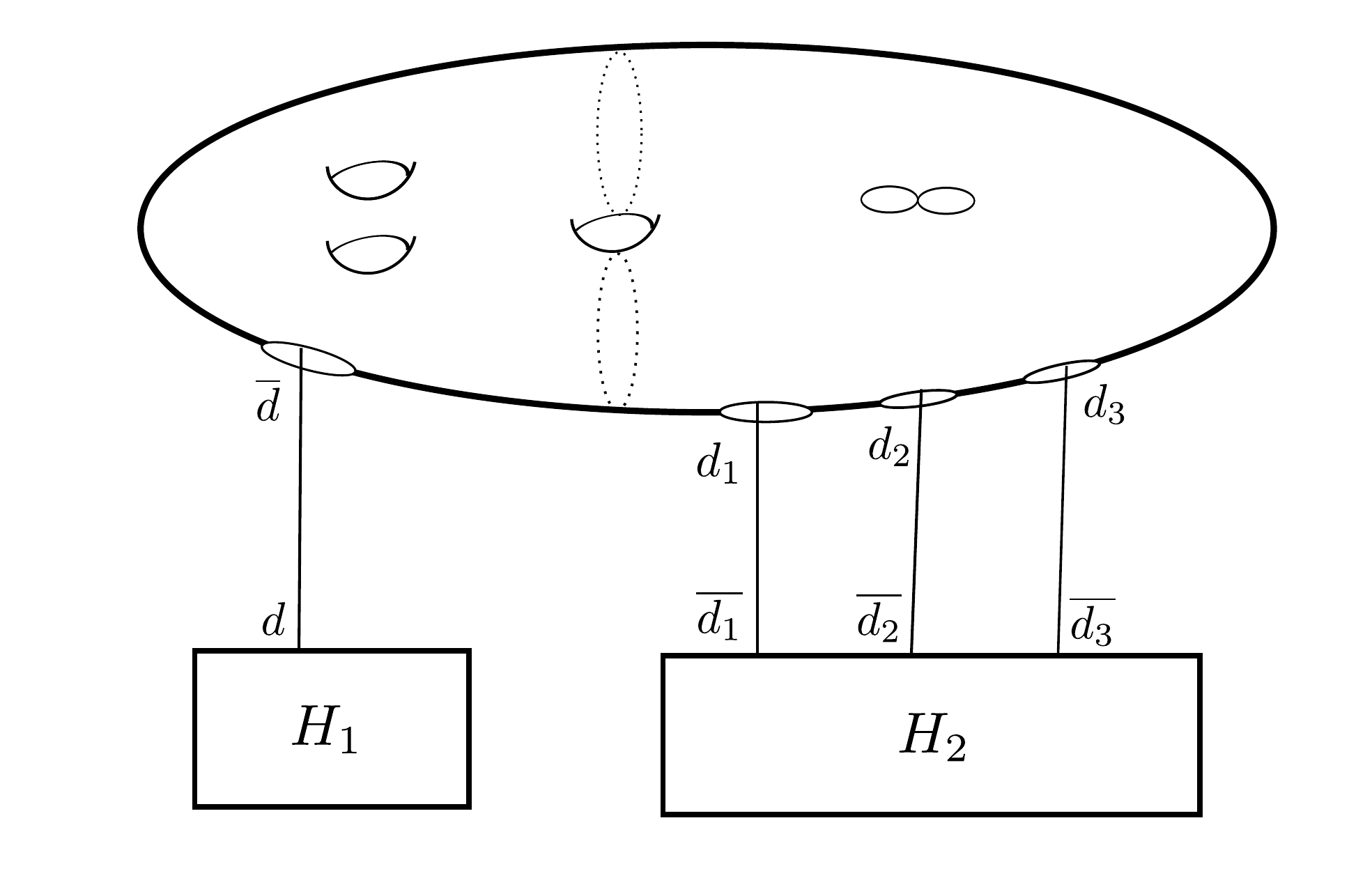}
\caption{If there is $d'\in H_1$ conjugate to $\overline{d}$ which is a product of two non trivial commutators, and $d'_1,d'_2,d'_3\in H_2$ conjugate to $\overline{d_1},\overline{d_2},\overline{d_3}$, such that $d'_1d'_2d'_3$ is a square, $\Sigma$ floor retracts.}
\end{figure}

\begin{proof} Suppose the three conditions hold. It will be convenient to use a slightly different notation : denote by $e_1, \ldots, e_n$ the edges of $\Lambda_{\Sigma}$, and suppose that $e_1, \ldots, e_{n_1}$ are adjacent to $H_1$, that $e_{n_1+1}, \ldots, e_{n_1+n_2}$ to $H_2$, and so on. We denote by $A$ the following set of indices: $\{n_1 + \ldots + n_k + 1 \mid k=0, \ldots, p\}$ (so that each $j \in A$ is the index of the "first" edge adjacent to some $H_i$).
	
For each $j$, we denote by $d_j$ the image of a generator of the edge group corresponding to $e_j$ in $S=\pi_1(\Sigma)$, and by $\bar{d}_j$ the image of that same generator in the corresponding vertex group $H_i$. 

A presentation for the group $G$ can be obtained from the group $S$, the groups $H_i$, and $n-r$ Bass-Serre elements $\{t_j \mid 1 \leq j \leq n, j \not \in A\}$, by identifying each $d_j$ to $\bar{d}_j$ if $j \in A$, and to $t_j\bar{d}_jt^{-1}_j$ if $j \not \in A$. 
 
From the first condition, we know that there are some elements $\delta_{e_1}=u_1 \bar{d}_1 u^{-1}_1, \ldots, \delta_{e_n}=u_n \bar{d}_n u^{-1}_n$ with $u_j$ in one of the factors of $H_1 * \ldots * H_p$ such that the product $\delta_{e_1} \ldots \delta_{e_n}$ can be written in $H_1 * \ldots * H_p$ as a product of $l' = \sum l_i$ commutators $[\bar{a}_k, \bar{b}_k]$ and $m' = \sum m_j $ squares $\bar{c}^2_j$. Note that  for $j \in A$ we can assume $u_j=1$, that is $\delta_{e_j} = \bar{d}_j$.

Denote by $S(l,m,n)$ the group presentation given by
$$\langle a_1, \ldots a_l, b_1, \ldots, b_l, c_1, \ldots c_m, d_1, \ldots, d_n \mid d_1\ldots d_n = \Pi^l_{i=1} [a_i,b_i] \Pi^m_{j=1} c^2_j \rangle$$ 
This is a presentation of a surface group corresponding to an $n$-punctured connected sum of $l$ tori and $m$ projective planes, where representatives of the boundary subgroups are generated by the elements $d_1, \ldots, d_n$.

Thus by the third condition $S$ admits such a presentation $S(l,m,n)$ for which $l \geq l'=\sum^{r}_{i=1} l_i$, and $m \geq m'=\sum^r_{j=1} m_j$.

We can thus define a morphism $r: G \to H_1 * \ldots * H_r$ which is the identity on each $H_i$, and sends
\begin{itemize}
\item $a_j, b_j$ to $\bar{a}_j, \bar{b}_j$ for $j \leq l'$ and to $1$ for $l' < j \leq l$;
\item $c_j$ to $\bar{c}_j$ for $j \leq m'$ and to $1$ for $m' < j \leq m$;
\item $t_j$ to $u_j$ for $j \not \in A$.
\end{itemize}
Note that the image of $S$ by $r$ is generated by the elements $\bar{a}_j, \bar{b}_j, \bar{c}_j$, and $u_j \bar{d}_j u^{-1}_j = \delta_{e_j}$.

If $l'>0$ it is easy to see that $r(S)$ is non abelian. This is also the case if the elements $\delta_{e_j}$ and $\bar{c}_j$ do not all commute (in particular, if $p>1$). By the second condition, this happens in particular whenever $l=0$.

In all these cases, we have found a retraction $r: G \to H_1 * \ldots * H_p$ which makes $(G,H_1 * \ldots * H_p , r)$ into a hyperbolic floor with associated decomposition $\Lambda_{\Sigma}$.

We are thus left to deal with the case where all the elements $\delta_{e_j}$ commute (so $p=1$), $l'=0$ and $l>0$. In this case we define a different morphism $r: G \to H_1$ (respectively $r: G \to H_1 *\Z$). The elements $\delta_{e_j}$ must be common powers of some element $z \in H_1$, say $\delta_{e_j} = z^{k_j}$. We see that the first condition implies that $K= k_1+ \ldots + k_n$ is $0$ if $m'=0$ and is even otherwise. 

We define $r$ as follows: pick an element $h \in H_1$ (respectively in $H_1 * \Z$) which does not commute to the $\delta_{e_j}$. Then let $r$ be the identity on the $H_i$ and let 
\begin{itemize}
\item $a_j \mapsto h$, $b_j\mapsto 1$ for all $j \leq l$;
\item$c_1 \mapsto z^{-K/2}$ and $c_i \mapsto 1$ for all $i>1$;
\item $t_j \mapsto u_j$ for $j>1$;
\end{itemize}
This is a morphism $r: G \to H_1$ or $G \to H_1 *\Z$ which is the identity on $H_1$ and for which $r(S)$ is non abelian. In the second case, extend it to $G*\Z$ by the identity on $\Z$. Then $(G,H_1 , r)$ is an (extended) hyperbolic floor with associated decomposition $\Lambda_{\Sigma}$.

Conversely, suppose that there exists a retraction $r: G \to H_1 * \ldots * H_p$ (respectively $G*\Z \to H_1*\Z$ if $p=1$ and $H_1$ is cyclic). We choose a maximal set of disjoint simple closed curves ${\cal C}$ whose corresponding elements are in the kernel of $r$, and build the the maps $r', \rho_{\cal C}$ and the graph of groups decomposition $\Gamma(\Sigma,{\cal C})$ for $\rho_{\cal C}(\pi_1(\Sigma))$ as described above. The vertex stabilizers of $\Gamma(\Sigma,{\cal C})$ are the fundamental groups of surfaces $\Sigma_1, \ldots , \Sigma_q$ each of which admits a presentation $S(l_k, m_k, n_k)$ (the surfaces $\Sigma_k$ are the connected components obtained by cutting $\Sigma$ along the curves of ${\cal C}$ and gluing disks on the boundary thus created). 

Because of the way the $\Sigma_k$ are obtained from $\Sigma$, we see that $\Sigma$ can be thought of as an $n$ punctured connected sum of $l$ tori with $m$ projective planes, with $l \geq l_1 + \ldots + l_q$, $m \geq m_1 + \ldots + m_q$ and $n =n_1 + \ldots + n_q$.  Moreover, by maximality of ${\cal C}$, the restriction of $r'$ to each $\pi_1(\Sigma_k)$ does not kill any element corresponding to a simple closed curve on $\Sigma_k$. In particular we know that all the boundary elements of $\Sigma_k$ are sent to the same $H_i$ by $r'$.

The images of the generators of the presentation $S(l_k, m_k, n_k)$ of $\pi_1(\Sigma_k)$ by $r'$ thus give us a way of writing a product of conjugates $\delta_{e^k_j}$ of these boundary elements as a product of $l_k$ non trivial commutators and $m_k$ non trivial squares in $H_i$. By grouping together all the surfaces with boundary elements sent to the factor $H_i$, we get exactly the first condition.

If $l=0$ and the elements $\delta_{e^k_j}$ all commute - say they all lie in a common maximal cyclic subgroup $\langle z \rangle$ - the fact that the image $r(S)$ is not abelian implies that at least one of the elements $c_j$ has image outside of $\langle z \rangle$, which means that the second condition is satisfied.
\end{proof}

\begin{defi} \label{SubsurfaceFloorRetractsDef} (retracting subsurface) Let $\Lambda$ be a graph of groups with surfaces with fundamental group $G$. Let $v$ be a surface type vertex with associated surface $\Sigma$, and let $\Sigma'$ be a subsurface of $\Sigma$. We say that \textbf{the subsurface $\Sigma'$ floor-retracts in $\Lambda$} if the surface $\Sigma'$ retracts in the graph of groups obtained by refining $\Lambda$ at $v$ by the graph of groups dual to the set of boundary curves of $\Sigma'$.
If $\Lambda$ is a graph of groups with sockets, $\Sigma$ a surface in $\Lambda$ and $\Sigma'$ a subsurface of $\Sigma$, we say $\Sigma'$ floor-retracts in $\Lambda$ if it floor-retracts in the corresponding graph of groups with surfaces.
\end{defi}

\begin{lemma} \label{RetractingSubsurfaceDifferentBoundaries} Let $\Lambda$ be a graph of groups with surfaces with fundamental group $G$. Let $v$ be a surface type vertex whose associated surface $\Sigma$ has nonempty boundary. 

If  a proper subsurface $\Sigma'$ of $\Sigma$ floor-retracts in $\Lambda$, there exists a subsurface $\tilde{\Sigma}$ which does not contain all the boundary components of $\Sigma$ and which also floor-retracts in $\Lambda$.
\end{lemma}

\begin{proof} Without loss of generality we may assume that $\Lambda=\Lambda_{\Sigma}$ and that it is bipartite between the vertex corresponding to $\Sigma$ and the non surface type vertices $u_1, \ldots, u_p$. Let $\Sigma'$ be a subsurface of $\Sigma$ containing all of the $n$ boundary components of $\Lambda$, and suppose $\Sigma'$ floor-retracts in $\Lambda$  .

After refining $\Lambda$ by the graph of groups for $\pi_1(\Sigma)$ dual to the boundary curves of $\Sigma'$, we get a graph of groups $\Lambda'$ with a vertex $v$ corresponding to $\Sigma'$, the vertices (coming from) $u_1, \ldots, u_p$ and vertices $u_{p+1}, \ldots, u_s$ corresponding to the connected components of $\Sigma-\Sigma'$.

Since $\Sigma'$ floor-retracts in $\Lambda'$, we can apply Lemma \ref{RetractingSurfaceEquations} to see that for each $j$, the product of some conjugates of the generators of the $n_j$ edge groups of $\Lambda'$ adjacent to $u_j$ can be written as a product of $l_j$ non trivial commutators and $m_j$ non trivial squares in the vertex group of $u_j$, such that $\pi_1(\Sigma')$ admits a presentation of the form $S(l',m',n')$ with $l' \geq \sum^s_{i=1} l_i$, and $m' \geq \sum^s_{j=1} m_j$. Moreover if $l=0$ then not all the squares and the conjugates commute.

The surface $\Sigma$ admits a presentation of the form $S(l,m,n)$ with $l \geq l'$ and $m \geq m'$. Because we assumed $\Sigma'$ to be proper, and to contain all the boundary components of $\Sigma$ we must have $l>l'$ or $m>m'$.

If $l>l'$, consider a subsurface $\tilde{\Sigma}$ of $\Sigma$ whose complement in $\Sigma$ is a pair of pants with boundary components $\beta_1, \delta_1, \delta_2$ where $\beta_1$  is a boundary component of $\Sigma$ which corresponds to an edge adjacent to $u_1$ and $\delta_1, \delta_2$ are a non boundary parallel curves on $\Sigma$.

To see that $\tilde{\Sigma}$ floor-retracts, note that the product $b^{h_1}_1 b^{h_2}_2 \ldots b^{h_{n_1}}_{n_1}$ of some conjugates of the generators $b_j$ of the edge groups adjacent to $u_1$ can be written as a product of $l_1$ commutators and $m_1$ squares in the vertex group $H_1$ associated to $u_1$ in $\Lambda$. 

We see that the vertex group $\tilde{H}_1$ of $\Lambda_{\tilde{\Sigma}}$ associated to the image $\tilde{u}_1$ of the vertex $u_1$ has presentation $\langle H_1, d_1, d_2 \mid b_1=d_1d_2 \rangle$ where $d_1, d_2$ are the elements corresponding to the edges adjacent to $\delta_1, \delta_2$. In particular, we see that the product $d^{h_1}_1d^{h_1}_2 b^{h_2}_2 \ldots b^{h_{n_1}}_{n_1}$, which is a product of some conjugates of the generators of the edge groups adjacent to $\tilde{u}_1$ in $\Lambda_{\tilde{\Sigma}}$, is equal to $ b^{h_1}_1 b^{h_2}_2 \ldots b^{h_{n_1}}_{n_1}$. Thus it can be written as the product of $l_1$ non trivial commutators and $m_1$ non trivial squares in the vertex group $\tilde{H}_1$.

Note that for $1<j\leq p$ we already know that the product of some conjugates of the generators of the $n_j$ edge groups adjacent to $u_j$ can be written as a product of $l_j$ commutators and $m_j$ squares in the vertex group of $u_j$. Also, if $l=0$ the conjugates of edge group generators and the squares do not all commute. Finally, we note that $\tilde{\Sigma}$ admits a presentation of the form $S(l-1, m, n+1)$. Since $\sum^r_{j=1} l_j \leq l' <l$ Lemma \ref{RetractingSurfaceEquations} applies, and we see that $\tilde{\Sigma}$ floor retracts.

If $l=l'$ we must have $m>m'$: in this case we set $\tilde{\Sigma}$ to be a non orientable subsurface of $\Sigma$ whose complement is a twice punctured projective plane with one boundary component $\beta_1$ in common with $\Sigma$, and the other a non boundary parallel curve of $\Sigma$. We proceed in a similar manner to see that $\tilde{\Sigma}$ floor-retracts. 
\end{proof}

\subsection{Conjugacy growth of orbits}

To show non homogeneity of a group $G$, it is enough to show that some tuple $u$ of an element of a proper elementary submodel $G_0$ has orbit under $\Aut(G_0)$ not contained in its orbit under $\Aut(G)$. For this we will need to be able to compare orbits. This section gives some results on the conjugacy growth of orbits which will help us with that task.

\begin{defi} (conjugacy growth) Let $G$ be a finitely generated group, and let $A$ be a subset of $G$. The conjugacy growth function $\lambda_A$ of $A$ with respect to a generating set $S$ of $G$ is defined by
$$ \lambda_A(n) = \abs{\{ [a] \mid a \in A \textrm{ and } [a]\cap B_n \neq \emptyset \}} $$
where $[a]$ denotes the conjugacy class of $a$ and $B_n$ is the ball of radius $n$ in $G$ with respect to the metric induced by $S$.
\end{defi}

We wish to prove the following
\begin{prop} \label{AutOfFreeProducts} Let $U$ be a torsion-free hyperbolic group which is not the free group of rank $2$, and which admits a free decomposition $U_1 * \ldots * U_t * \F$ in which none of the factors $U_i$ are cyclic, and $\F$ is a free group. Assume the decomposition is non trivial, that is, that either $t>1$ or $\F$ is non trivial, and let $u \in U$ be an element which is not contained in a conjugate of one of the factors $U_i$. Then the conjugacy growth of the orbit of $u$ under $\Aut(U)$ is at least exponential.
\end{prop}

We now prove a series of lemmas about growth of certain subsets of torsion free hyperbolic groups. The first concerns double cosets of maximal cyclic subgroups.
\begin{lemma} \label{GrowthDoubleCoset} Let $\Gamma$ be a torsion free hyperbolic group with a fixed generating set $S$. Fix two elements $a$ and $h$ in $\Gamma$, assume $a \neq 1$, and denote by $Z(a)$ the centralizer of $a$ in $\Gamma$. There are constants $C,D$ such that for any $g \in \Gamma$, we have
	$$\abs{hZ(a)g \cap B_n}, \abs{gZ(a)h \cap B_n} \leq C(n + \abs{h}) + D $$
where $B_n$ is the ball of radius $n$ in $\Gamma$ with respect to $S$.
\end{lemma}

\begin{proof} Cyclic subgroups of hyperbolic groups grow linearly, hence there exist strictly positive constants $C_a, D_a$ (depending only on $a$) such that for $n$ large enough we have 
	$$\abs{B_n \cap Z(a)} \leq C_an +D_a$$
	
	Let $z_0 \in Z(a)$ be such that $u_0 =hz_0g$ is the shortest element in $hZ(a)g$. Let $g_0 = h^{-1}u_0 = z_0g$ - note that $hZ(a)g = hZ(a)g_0$. If $hZ(a)g \cap B_n$ is non empty, then in particular $\abs{u_0} \leq n$ so $\abs{g_0} \leq \abs{h} + n$. Now if $u \in hZ(a)g \cap B_n$, then  $u=hzg_0$ for some $z \in Z(a)$, hence $z=h^{-1}ug^{-1}_0$ so $\abs{z} \leq \abs{h} + n + \abs{g_0} \leq 2(\abs{h} +n)$. There are at most $C_a(2(\abs{h} +n)) +D_a$ elements $z$ of $Z(a)$ which satisfy this, so we get
	$$ \abs{hZ(a)g \cap B_n} \leq  C_a(2(\abs{h} +n)) +D_a$$
The proof for the set $\abs{gZ(a)h \cap B_n}$ is similar.
\end{proof}

\begin{lemma} Let $\Gamma$ be a non cyclic torsion free hyperbolic group, fix $S$ a finite generating set. Let $a \in \Gamma$, and denote by $Z(a)$ the centralizer of $a$ in $\Gamma$. Consider a set $R_{Z(a)}$ of shortest (with respect to $S$) representatives of the cosets $hZ(a)$ for $h \in \Gamma$. Then $R_{Z(a)}$ has exponential growth.
\end{lemma}

\begin{proof} Applying Lemma \ref{GrowthDoubleCoset} with $h=1$, we get that there exists constants $C, D$ depending only on $a$ such that for any  $g \in R_{Z(a)}$  then $\abs{ gZ(a) \cap B_n} \leq Cn +D$.

We can write $\Gamma$ as the disjoint union of the cosets $gZ(a)$ for $g \in R_{Z(a)}$. Since the growth of each of these cosets is bounded by a common linear function, and since $\Gamma$ has exponential growth, we must have that the number of these cosets meeting the ball of radius $n$ grows exponentially with $n$. But this implies that the growth of $R_{Z(a)}$ is exponential. 
\end{proof}

Let $G$ be a non cyclic torsion free hyperbolic group, and suppose $G= A*B$. For any $g \in A$, denote by $\theta_g$ the automorphism of $G$ which restricts to the identity on  $A$ and to conjugation by $g$ on $B$.

\begin{lemma} Let $G$ be a torsion free hyperbolic group, and suppose $G= A*B$ with $A$ non cyclic. Let $u \in G$ be non elliptic in this splitting.

Then the set $O=\{ \theta_g(u) \mid g \in A, g \neq 1\}$ has exponential growth.
\end{lemma}

\begin{proof} Note that $\theta_g(u) = \theta_{g'}(u)$ iff $g=g'$ (by a normal form argument). Now there exist constants $C,D$ depending on $u$ only such that $\abs{\theta_g(u)} \leq C\abs{g} +D$

For any $g \in A \cap B_{(n-D)/C}$ we have that $\abs{\theta_g(u)} \leq n$ so $\abs{O \cap B_n} \geq \abs{A \cap B_{(n-D)/C}}$, which grows exponentially with $n$. 
\end{proof}

We can now prove Proposition \ref{AutOfFreeProducts}:
\begin{proof} Note that each factor $U_i$ of the decomposition of $U$ is hyperbolic (as a retract of a hyperbolic group). Grouping the factors together, we can find a decomposition $U = A *B$ where $u$ is not elliptic, and $A$ is not cyclic. Let $S_A, S_B$ be generating sets for $A,B$ respectively and let $S=S_A \cup S_B$.

We will show that the set $O= \{ \theta_g(u) \mid g \in A\}$ has exponential conjugacy growth with respect to $S$, this is enough to prove the result. Assume without loss of generality that $u=a_1b_1 \ldots a_lb_l$ with $l \geq 1$ and $a_i, b_i$ non trivial elements from $A,B$ respectively.

Note first that $\theta_g(u)$ and $\theta_{g'}(u)$ are conjugate if and only if $\theta_{g'g^{-1}}(u)$ is conjugate to $u$ (using for example that for any $g,h \in A$ we have $\theta_g \circ \theta_{h} = \theta_{hg}$ and $\theta^{-1}_{h} = \theta_{h^{-1}}$). Setting $h=g'g^{-1}$, we see that $\theta_h(u)$  is conjugate to $(h^{-1}a_1 h) b_1 (h^{-1}a_2 h) \ldots (h^{-1}a_lh) b_l$, and we can assume without loss of generality that this is a shortest normal form (with respect to the splitting $A*B$) in the conjugacy class of $\theta_h(u)$. We see by a normal form argument that if $\theta_h(u)$ is conjugate to $u$, we must have $h^{-1} a_i h = a_j$ for some $i,j$. Note now that for any $i, j$, two solutions to the equation $xa_ix^{-1}=a_j$ differ by an element of $Z(a_i)$. If there exists such a solution, pick one and denote it by $h_{ij}$. We see that if $\theta_{g'}(u)$ is conjugate to $\theta_{g}(u)$ then $g'g^{-1}$ lies in one of the cosets $h_{ij}Z(a_i)$, hence $g'$ lies in $h_{ij}Z(a_i)g$.

By Lemma \ref{GrowthDoubleCoset}, there exist constants $C_1, D_1$ such that for any $i,j$ and $g \in A$, the growth of $\abs{h_{ij}Z(a_i)g}$ is bounded by the linear function $C_1n + D_1$.

Suppose $x$ lies in $O \cap [\theta_g(u)] \cap B_n$, then $x = \theta_{g'}(u)$ for some $g'$ in one of the double cosets $h_{ij}Z(a_i)g$. Moreover, we must have $\abs{g'} \leq \abs{x} + K \leq n +K$ for some constant $K$ depending only on $u$: indeed, one can see for example that $$\abs{\theta_{g'}(u)} = \abs{(a_1g')b_1 ((g')^{-1}a_2g') \ldots ((g')^{-1}a_lg')b_l (g')^{-1}} \geq \abs{a_1g'} \geq \abs{g'} - \abs{a_1}$$
so taking $K = \abs{a_1}$ proves the claim (note that here our specific choice of generating set comes into play). 

Hence we get
$$\abs{O \cap [\theta_g(u)] \cap B_n} \leq \abs{h_{ij}Z(a_i)g \cap B_{n+K}} \leq C_1(n+K) +D_1$$
that is, the growth of $O \cap [\theta_g(u)]$ is bounded by a linear function whose constants depend only on $u$. 

Since $O$ has exponential growth, and since the growth of the intersection of $O$ with any conjugacy class is bounded by a common linear function, we see that $O$ has exponential conjugacy growth.
\end{proof}

\begin{lemma} \label{AutOfSurfaceGroup} Let $G$ be a freely indecomposable, torsion free hyperbolic group, and let $\Lambda$ denote its maximal cyclic JSJ decomposition. 
Let $g$ be an element corresponding to a simple closed curve on a surface of $\Lambda$. Then the conjugacy growth of the orbit of $g$ under $\Aut(G)$ is at most polynomial.
\end{lemma}

Let $\Sigma$ be a hyperbolic surface. It is a theorem of \cite{RivinSCC} that there exists a polynomial $P$ such that the number of closed simple geodesic curves on $\Sigma$ of length at most $L$ is bounded above by $P(L)$. Thus the conjugacy growth of the set of elements of $\pi_1(\Sigma)$ representing a simple closed curve of $\Sigma$ is at most polynomial.

\begin{proof} On a surface vertex group $S$ of $\Lambda$, elements of $\Mod(G)$ restrict to the composition of a conjugation together with a modular automorphism of $S$. In particular, an element $g$ corresponding to a simple closed curve on the surface is sent to a conjugate in $G$ of an element of $S$ also corresponding to a simple closed curve. Hence the orbit of $g$ under $\Mod(G)$ has polynomial conjugacy growth. Since $\Mod(G)$ has finite index in $\Aut(G)$, the same is true of the orbit of $g$ under $\Aut(G)$.
\end{proof}

\subsection{A first obstacle to homogeneity: proper subsurfaces floor-retracting}

We can now formulate the first obstruction to homogeneity of a torsion-free hyperbolic group

 \begin{prop} \label{SubsurfaceRetractNonHomogeneous} Let $G$ be a freely indecomposable torsion free hyperbolic group which is not the fundamental group of a closed surface of characteristic $-2$ or $-1$, and denote by $\Lambda$ its JSJ decomposition. Suppose some surface $\Sigma$ of $\Lambda$ has a proper subsurface $\Sigma'$ which floor-retracts in $\Lambda$. Then $G$ is not homogeneous.
\end{prop}

\begin{proof} Non homogeneity of surface groups associated to surfaces of characteristic at most $-3$ is proven in \cite{PerinSklinosHomogeneity}. We may thus assume that $G$ is not the fundamental group of a closed surface, in particular, $\Sigma$ has nonempty boundary. Recall that $\Lambda_{\Sigma}$ denotes the graph of groups obtained by collapsing all edges of $\Lambda$ which are not adjacent to the vertex $v$ corresponding to $\Sigma$.

Let $\Delta(S, {\cal C})$ be the graph of groups of fundamental group $S =\pi_1(\Sigma)$ which is dual to the set of boundary curves ${\cal C}$ of $\Sigma'$ which are not parallel to the boundary components of $\Sigma$. Denote by $\Lambda'$ the graph of groups obtained from $\Lambda_{\Sigma}$ by refining $v$ by $\Delta(S, {\cal C})$. Note that since $\Sigma'$ is a proper subsurface, the graph of groups $\Delta(S, {\cal C})$ is not reduced to a point.

The fact that $\Sigma'$ floor-retracts in $\Lambda$ means that there is a structure of (extended) hyperbolic floor $(G, G',r)$ for $G$, whose associated free product decomposition $G'=G_0*G_1* \ldots * G_m$ is such that the factors $G_i$ are the fundamental groups of connected components of the complement of the vertex associated to $\Sigma'$ in $\Lambda'$.

\begin{figure}[ht!] \label{SubSurfaceRetractsFig}
	\centering
	\includegraphics[width=.9\textwidth]{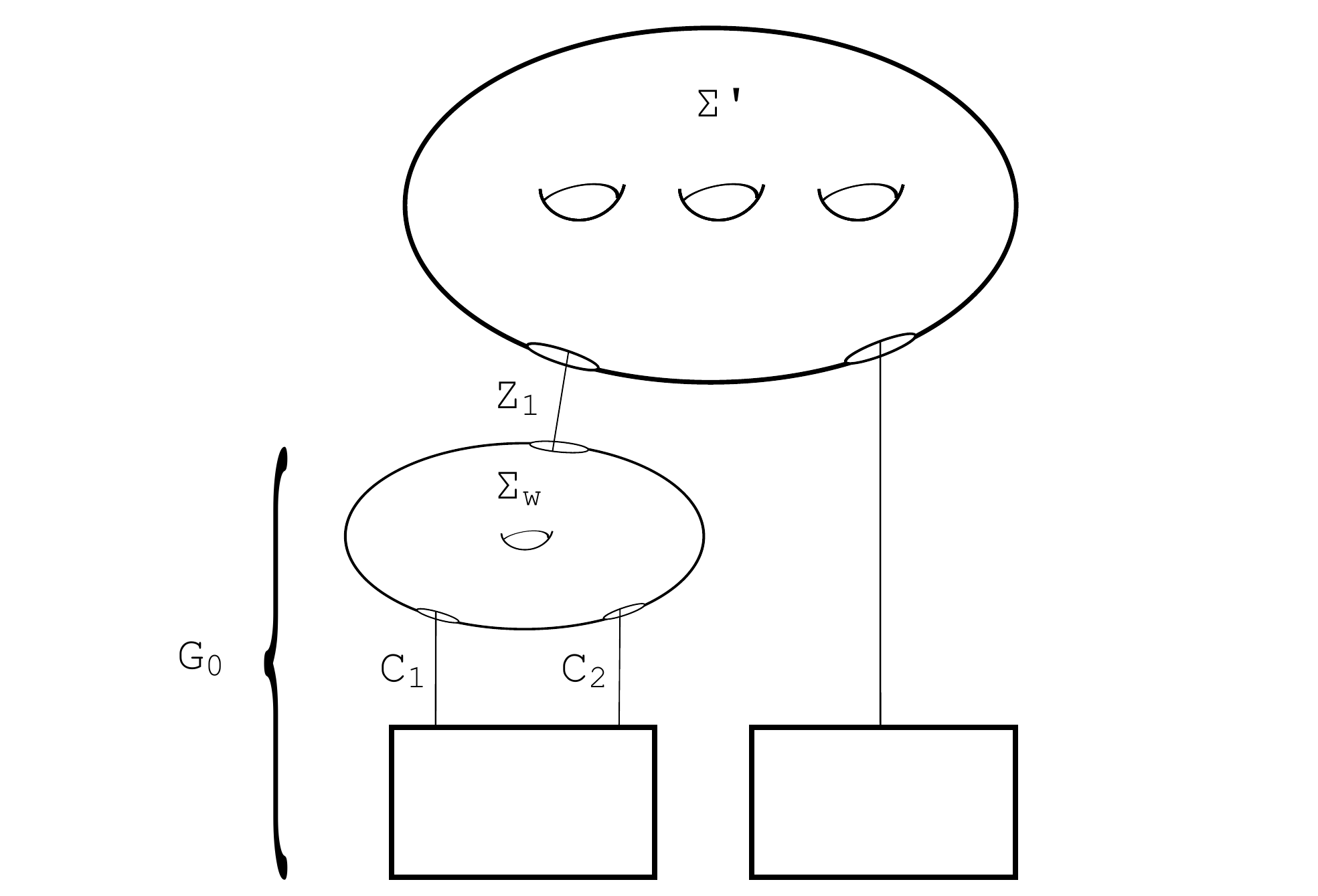}
	\caption{The subsurface $\Sigma'$ retracts}
\end{figure}

Note that by Lemma \ref{RetractingSubsurfaceDifferentBoundaries}, we may assume that there is a component of $\Sigma - \Sigma'$ which contains a boundary component of $\Sigma$. Let $w$ be a vertex of $\Delta(S, {\cal C})$ corresponding to such a subsurface $\Sigma_w$, without loss of generality assume that the fundamental group of the connected component in which $w$ lies is $G_0$ (see Figure \ref{SubSurfaceRetractsFig}). Choose $Z_1, \ldots, Z_p$ representatives of the conjugacy classes of edge groups of $\pi_1 (\Sigma_w)$ whose corresponding edges are adjacent to the vertex of $\Sigma'$, and $C_{1}, \ldots, C_q$ representatives for the other classes of boundary subgroups (note that $p>0$ and $q>0$ by our assumption on $\Sigma'$). Relatively to $Z_{2}, \ldots, Z_p, C_1, \ldots, C_q $, the group $S_w = \pi_1(\Sigma_w)$ admits a decomposition of the form
$$Z_2 * \ldots *Z_p*C_1* \ldots *C_q *F \; (\dagger)$$
(where $F$ is a free group),  while relatively to $Z_{1},\ldots, Z_p, C_1, \ldots, C_q$ it doesn't admit any non-trivial free decomposition.

The decomposition $(\dagger)$ induces a non trivial free decomposition $U_1 * \ldots * U_r * \F_s$ of $G_0$: indeed, if $\Lambda'_0$ denotes the subgraph of subgroups of $\Lambda'$ containing $w$ whose fundamental group is $G_0$, it can be refined at $w$ by the free product decomposition of $S_w$ relative to $Z_{2}, \ldots, Z_p, C_1, \ldots, C_q$ since the edge groups adjacent to $w$ in $\Lambda'_0$ correspond exactly to $C_{1}, \ldots, C  _q$. Moreover, in this decomposition of $G_0$, the subgroup $Z_1$ is not elliptic. Let $z_1$ denote a generator for $Z_1$.

Suppose first that $G_0$ is not a free group of rank $2$. By Proposition \ref{AutOfFreeProducts}, the orbit of $z_1$  under $\Aut(G_0)$ has exponential conjugacy growth in $G_0$, therefore it has exponential conjugacy growth in $G$ since elements of $G_0$ which are conjugate in $G$ are conjugate in $G_0$ ( $G_0$ is a retract of $G$). On the other hand, by Lemma \ref{AutOfSurfaceGroup}, the orbit of $z_1$ under $\Aut(G)$ has polynomial conjugacy growth in $G$. This implies that there exists an element $y_1$ of $G_0$ which is in the orbit of  $z_1$  under $\Aut(G_0)$ but not in its orbit under $\Aut(G)$. By Proposition \ref{CriterionForNonHomogeneity}, $G$ is not homogeneous.

To deal with the case where $G_0 \simeq \F_2$ we need a slightly more detailed analysis. Note that in general if the surface $\Sigma_w$ associated to $w$ admits a presentation of the form $S(g,m,p+q)$, then the group $G_0$ is of the form
$$H_1 * \ldots * H_k * Z_2 * \ldots *Z_p* \F_{2g} * \F_m * \langle t_1 \rangle * \ldots *\langle t_{q-k} \rangle $$
where the factors $H_i$ are the groups of the non surface vertices of $\Lambda'$ adjacent to $w$. 

Note that  we have $k \geq 1$ and because $\Sigma_w$ is not a cylinder, either $g \neq 0$, or $m \neq 0$, or $p>1$, or $q>1$.
In particular, if $G_0$ is free of rank $2$, there are four   possible cases
 \begin{enumerate}
\item $k=2, p=1, q=2, g=m=0$ (then $\Sigma_w$ is a pair of pants and $H_1 \simeq H_2 \simeq \Z$)
\item $k=1, p=2, q=1, g=m=0$  (then $\Sigma_w$ is a pair of pants and $H_1 \simeq \Z$);
\item $k=1, p=1, q=2, g=m=0$  (then $\Sigma_w$ is a pair of pants and $H_1 \simeq \Z$);
\item $k=1, p=1, q=1, g=0, m=1$  (then $\Sigma_w$ is a twice punctured projective plane and $H_1 \simeq \Z$). 
 \end{enumerate}
We see that in all these cases, the boundary subgroup $C_{1}$ is generated by a power $z=b^k$ of a primitive element $b$ of $G_0$. The set of conjugacy classes of primitive elements of $\F_2$ is infinite, hence the orbit $\{ \beta^k \mid \beta \textrm{ primitive }\}$ of $b$ in $G_0$ meets infinitely many conjugacy classes. On the other hand, $C_{1}$ is an edge group of $\Lambda$, hence by Lemma \ref{OrbitEdgeElement} the orbit of $z$ under $\Aut(G)$ consists of finitely many conjugacy classes.
\end{proof}

\section{Second obstruction to homogeneity: JSJ decompositions of tower retracts refine} \label{JsjDecompositionsOfFloorRetracts}

\label{GrowingJSJSection}

In this section, we exhibit another obstacle to homogeneity of a freely indecomposable torsion free hyperbolic group $G$. In view of the previous section, for a group to be homogeneous, it is necessary that the decomposition $\Gamma$ associated to any of its floor structure $(G,G', r)$ corresponds to a floor-retracting subset $\Sigma_1, \ldots, \Sigma_p$ of full surfaces of the JSJ decomposition $\Lambda$ of $G$. In this section, we show that moreover, it is necessary that the JSJ decompositions of the non cyclic free factor of $G'$
are exactly the connected components of the complement in $\Lambda$ of the set of vertices corresponding to $\Sigma_1, \ldots, \Sigma_p$.

\begin{prop} \label{InfiniteConjClasses} Let $G$ be a torsion-free hyperbolic group, and let $g \in G$. Assume moreover that if $G$ is isomorphic to $\F_2 = \F(a,b)$ then $g$ is not conjugate to an element of $\langle [a,b] \rangle$. If there exists a non trivial free splitting of $G$ or a splitting over a maximal cyclic group in which $g$ is not elliptic, then the orbit of $g$ under $\Aut(G)$ contains infinitely many non conjugate elements.
\end{prop}

\begin{proof} If $G$ is not isomorphic to $\F_2$ and admits a free product decomposition in which $g$ is not elliptic, the result follows straight from Proposition \ref{AutOfFreeProducts}.

If $G \simeq \F_2 = \langle a,b \mid \rangle$ and $g$ is not conjugate into $\langle [a,b] \rangle$, we think of $G$ as the fundamental group of a punctured torus whose boundary subgroups are the conjugates of $\langle [a,b] \rangle$. By \cite[Remark 4.3]{BestvinaHandelTrainTracksFree}, 
if $f$ is an automorphism of $G$ corresponding to a pseudo-Anosov homeomorphism, then $\{f^k(g) \mid k \in \Z\}$ meets infinitely many conjugacy classes.

If $G$ admits a splitting as an amalgamated product $A*_CB$ over a maximal cyclic group $C = \langle c \rangle$ in which $g$ is not elliptic, we let $\tau$ be the Dehn twist by $c$. Up to conjugation, $g$ can be written in a normal form as 
$$g = a_1b_1\ldots a_kb_k$$ with $k \geq 1$, $a_i \in A\setminus C, b_i \in B\setminus C$ for all $i$. Then 
$$\tau^n(g) =  (a_1c^n) b_1 (c^{-n}a_2c^n)\ldots (c^{-n}a_k c^n)b_kc^{-n}$$ 
is also in normal form of length $k$. Since $g$ and $\tau^n(g)$ have the same length, if they are conjugates then by Theorem 2.8 of Chapter IV in \cite{LyndonSchupp} there are $i_n,l_n$ such that 
$$\tau^n(g)=c^{l_n}a_{i_n}b_{i_n}\cdots a_{i_n-1}b_{i_n-1}c^{-l_n}$$ 
(where indices are modulo $k$).
If for infinitely many $n$ the image $\tau^n(g)$ is conjugate to $g$ then there is an infinite subsequence for which $i_n=i$.
Assume then that $m,n$ are indices in this subsequence: $$\tau^n(g)=c^{l_n}a_ib_i\cdots a_{i-1}b_{i-1}c^{-l_n}$$ and 
$$\tau^m(g)=c^{l_m}a_ib_i\cdots a_{i-1}b_{i-1}c^{-l_m}$$ 
Now 
$$\tau^m(g)=\tau^{m-n}(\tau^n(g))=c^{l_n}a_ic^{m-n}b_ic^{n-m}\cdots a_{i-1}c^{m-n}b_{i-1}c^{n-m-l_n}$$ so
$$a_ib_i\cdots a_{i-1}b_{i-1}=(c^{l_n-l_m}a_ic^{m-n})b_i(c^{n-m}a_{i+1} c^{m-n})\cdots (c^{n-m}a_{i-1}c^{m-n})(b_{i-1}c^{n-m+l_m-l_n})$$
and by a normal form theorem (see Theorem 2.6 of Chapter IV in \cite{LyndonSchupp}) we have that:
$$b_{i-1}c^{l_n-l_m+m-n}b_{i-1}^{-1}\in C$$
Now $C$ is maximal cyclic in a torsion-free hyperbolic group, so
it is malnormal. Since $b_{i-1}\notin C$, we must have that $c^{l_n-l_m+m-n}$ is trivial, so $l_n-l_m=n-m$. By the same normal form argument, this implies that 
$$a_{i-1}c^{n-m}a_{i-1}^{-1}\in C$$ and since $a_{i-1}\notin C$ we have that $c^{n-m}=1$ so $n=m$ - a contradiction.
A similar argument proves the claim when $G$ is an HNN extension.

\end{proof}

\begin{prop} \label{GrowingJSJNonHomogeneous} 
Let $G$ be a freely indecomposable torsion free hyperbolic group with maximal-cyclic JSJ decomposition 
$\Lambda$. Let $\{\Sigma_1, \ldots, \Sigma_m\}$ be a set of surfaces of $\Lambda$ which floor-retracts, with corresponding vertices $\{v_1, \ldots, v_m\}$. Let $G_0$ be the fundamental group of some connected component $\Lambda_0$ of $\Lambda - \{v_1, \ldots, v_m\}$. Suppose one of the following cases occurs:
\begin{enumerate}
\item  $G_0$ is freely decomposable and not isomorphic to $\F_2$;
\item  $G_0 \simeq \F(a,b)$ and at least one of the edge groups of $\Lambda - \Lambda_0$ adjacent to $\Lambda_0$ is not conjugate to $\langle [a,b] \rangle$;
\item $\Lambda_0$ is not the maximal-cyclic JSJ decomposition of $G_0$.
\end{enumerate}

Then $G$ is not homogeneous.
\end{prop}

\begin{proof} Note that $G_0$ is itself torsion-free hyperbolic since it is a retract of $G$. By Proposition \ref{RelativeJSJSubgraphOfGroups}, $\Lambda_0$ is the maximal cyclic JSJ of $G_0$ relative to the collection of its adjacent edge groups $(H_1, \ldots, H_r)$ in $\Lambda$.

If $G_0$ is freely decomposable and not isomorphic to $\F_2$, then one of these edge group $H_i= \langle z_i \rangle$ is not elliptic in the Grushko decomposition of $G_0$ (otherwise $G$ itself would be freely decomposable).  Thus the orbit of $z_i$ under $\Aut(G_0)$ contains infinitely many non conjugate elements by Proposition \ref{InfiniteConjClasses}.

If $G_0 \simeq \F(a,b)$ and at least one of the edge groups $\langle z \rangle$ of $\Lambda - \Lambda_0$ adjacent to $\Lambda_0$ is not conjugate to an element of $\langle [a,b] \rangle$, then the orbit of $z$ under $\Aut(G_0)$ contains infinitely many non conjugate elements by Proposition \ref{InfiniteConjClasses}.

If $G_0$ is freely indecomposable, but $\Lambda_0$ is not the maximal cyclic JSJ of $G_0$, then by Remark \ref{JSJVsRelJSJ} there exists a maximal cyclic splitting of $G_0$ in which one of the groups $H_i = \langle z_i \rangle$ is not elliptic. Similarly, Proposition \ref{InfiniteConjClasses} implies that the orbit of $z_i$ under $\Aut(G_0)$ contains infinitely many pairwise non conjugate elements.

Since $G_0$ is a retract of $G$, non conjugate elements in the orbit of $z_i$ under $\Aut(G_0)$ are also non conjugate in $G$. But by Proposition \ref{OrbitEdgeElement}, the orbit of $z_i$ under $\Aut(G)$ consists of finitely many conjugacy classes. Hence not every automorphism of $G_0$ extends to an automorphism of $G$.

Now since $\{\Sigma_1, \ldots, \Sigma_m\}$ floor-retracts, there is a hyperbolic floor structure $(G, G', r)$ where $G_0$ is a free factor of $G'$. By Proposition \ref{CriterionForNonHomogeneity}, the group  $G$ is not homogeneous.
\end{proof}

\section{Structure of cores in the absence of the first two obstructions} 

We have now seen two possible obstructions to homogeneity of a group: the presence of retracting subsurfaces in its JSJ decomposition, and the possibility that the JSJ or Grushko decompositions of fundamental groups of the connected components obtained by removing a floor-retracting set of surfaces refines. To complete the characterization of homogeneity we need to describe cores (recall Definition \ref{RelativeCoresDef}) of $G$ when neither of these two obstructions occur. 

For that purpose we need to understand how surfaces in some splitting of $G$, and floor-retracting surfaces in particular, appear in the maximal cyclic JSJ decomposition of $G$.

First we prove a generalization of Lemma 4.1 in \cite{BestvinaFeighnOuterLimits}. We define
\begin{defi} Given spaces $U_1, \ldots, U_m$ and an integer $n$, we call star graph on  $U_1, \ldots, U_m$ and $n$ and denote by $ST(U_1, \ldots, U_m, n)$ the space built by taking a graph consisting of one central vertex $v$, vertices $v_1, \ldots, v_m$ adjacent to $v$, and $n$ loops at $v$, and identifying each $v_l$ to a point in the space $U_l$.
\end{defi}
 The fundamental group of $ST(U_1, \ldots, U_m, n)$ is of course $\pi_1(U_1)* \ldots *\pi_1(U_m)*\F_n$.

\begin{lemma} \label{GeneralBFLemma}Let $X_1, \ldots, X_k$ be connected spaces, of the form 
$$X_i = ST(U^i_1, \ldots, U^i_{m_i}, n_i)$$
and let $\Sigma$ be a connected surface with boundary. Let $f: \partial\Sigma \to  \bigsqcup X_i$ be an immersion that is essential on each boundary component, and denote by $X$ the adjunction space $\Sigma \cup_f \bigsqcup_i X_i$.

Denote by $Y$ the space $ST(U^1_1, \ldots, U^1_{m_1}, \ldots, U^k_1, \ldots, U^k_{m_k}, n)$.
Suppose there is a homotopy equivalence $F: X \to Y$ which restricts on each $X_i$ to an immersion $X_i \to Y$ which send each $U^i_j$ in $X_i$ homeomorphically to the copy of $U^i_j$ in $Y$.

Then up to homotopy equivalence, either the image of $\Sigma$ by $F$ is contained in one of the $U^i_j$, or the attaching map $f: \partial\Sigma \to  \bigsqcup X_i$ sends one of the boundary component of $\partial \Sigma$ to one of the $n_i$ loops attached to the central vertex of one of the spaces $X_i$, and all the other boundary components in the complement in $\bigsqcup X_i$ of this loop. 
\end{lemma}

\begin{proof} We proceed just as in the proof of Lemma 4.1 in \cite{BestvinaFeighnOuterLimits}. It is enough to show that for some $i$, some edge of $X_i$ not lying in any $U^i_j$ is crossed exactly once by the image of the map $f: \partial\Sigma \to \bigsqcup X_i$. For this, consider the inverse image by $F$ of the midpoints of edges of $Y$. Assuming $F$ is transverse and minimal, this gives an embedded graph in $X$ which consists of loops in the interior of (the image of) $\Sigma$, arcs with endpoint in the  (the image of the) boundary of $\Sigma$, and isolated vertices (on edges of the spaces $X_i$). If the graph has a valence one vertex, this corresponds to an edge of $Y$ which is crossed exactly once by the image of the map $F \circ f$ and we are done. 

If the image by $F$ of $\Sigma$ in $Y$ is not contained in one of the subspaces $U^i_j$ up to homotopy equivalence, the graph must meet the image of $\Sigma$ in $X$ - in particular it does not consist only of isolated points. 
Assume there are no valence one vertices in this graph - then it must contain a loop. We now prove that this loop is non null homotopic: this will contradict the $\pi_1$-injectivity of $F$ and we will be done. If the loop lies in the interior of $\Sigma$, it must be non null homotopic, otherwise $F$ is not minimal. If the loop consists of several arcs, we look at a lift of it in the universal cover of $X$ to see that the only way it can lift to a loop is if it doesn't cross to a different copy of $\tilde{\Sigma}$. This implies it is made of a single arc which together with a piece of the boundary of $\Sigma$ bounds a disk in the image of $\Sigma$ in $X$. Hence the image in $Y$ of this piece of boundary is a loop which bounds a disk in $Y$. Since the map $X \to Y$ is a $\pi_1$-embedding, it also bounds a disk in $X$. This contradicts our choice of the attaching map $f$. 
\end{proof}

From this we deduce the following lemma :
\begin{lemma} \label{SurfaceInFreeComponent} Let $G$ be a torsion-free hyperbolic group. Let $G=G_1*\cdots *G_m * \F$ be the Grushko decomposition of $G$ relative to a tuple $h$, where the factors $G_i$ are freely indecomposable and non cyclic (relative to $h$ for $i=1$), and $\F$ is a possibly trivial free group. Let $(G, G',r)$ be a hyperbolic floor structure for $G$ relative to $h$, and assume the corresponding floor decomposition $\Gamma$ admits a unique surface $\Sigma$. 

Then there exists $i$ in $\{1, \ldots, m \}$ such that some conjugate of $G_i$ contains the surface group associated to $\Sigma$, and for which there exists a floor structure $(G_i, G'_i, r_i)$ whose associated decomposition $\Gamma_i$ has exactly one surface corresponding to $\Sigma$.
\end{lemma}
The proof is similar to that of Lemma 5.19 of \cite{PerinElementary}, which shows that free groups do not admit hyperbolic floor structures.

\begin{proof}
Consider the action of each $G_i$ on the tree $T_{\Gamma}$ corresponding to $\Gamma$. 

If $G_i$ does not intersect a surface type vertex group in more than a boundary subgroup, the fact that $G_i$ is (relatively) freely indecomposable together with $1$-acylindricity near surface type vertex (see Remark \ref{1Acylindricity}) and the fact that $h$ is elliptic if $i=1$ implies that $G_i$ must in fact be elliptic in $\Gamma$.

Suppose on the other hand there is a surface type vertex $u$ of $T_{\Gamma}$ with stabilizer $S$, such that $S_i=G_i\cap S$ is not contained in a boundary subgroup. If $S_i$ is of infinite index in $S$, then by Lemma 3.11 in \cite{PerinElementary} it is the free product of edge groups with a (non-trivial) free group, and so $G_i$ inherits a non-trivial free splitting from its action on $T_{\Gamma}$. But $G_i$ is freely indecomposable, so $S_i$ must have finite index in $S$. Now if $S_i\neq S$ and $g\in S\setminus S_i$ then $gS_ig^{-1}$ is also of finite index in $S$ and so is $S_i \cap gS_ig^{-1}$, but this is a contradiction: if $s\in S_i\cap gS_ig^{-1}$ it stabilizes two distinct vertices in the tree corresponding to the free splitting $G_1*\cdots *G_m*\F$, and so it stabilizes the path between them. But edge stabilizers of this tree are trivial. Hence we must have $S_i = S$, that is, $S \leq G_i$.

Note moreover that the second possibility occurs for at most one of the factors $G_i$. 

Let us assume that it does not occur for any of the the factors $G_i$ and get a contradiction. By definition of a floor structure, $G' = H_1 * \ldots *H_l$ where the groups $H_j$ are non surface type vertex groups of $\Gamma$. Now since the groups $G_i$ are elliptic in $\Gamma$, the free decomposition inherited by $H_j$ from the Grushko decomposition of $G$ is a free product of conjugates of the groups $G_i$ together with a free group. This must in fact be the Grushko decomposition of $H_j$ since the groups $G_i$ are freely indecomposable. Moreover, if $G^{\alpha}_i$ and $G^{\beta}_i$ are both subgroups of $H_j$, the retraction $r$ restricts to the identity on both these subgroups. Thus for any $g \in G_i$, we have $\alpha g \alpha^{-1}= r(\alpha g \alpha^{-1}) = r(\alpha) r(g) r(\alpha^{-1})$ so $r(g)= r(\alpha)^{-1} \alpha g \alpha^{-1} r(\alpha)$, and similarly  $r(g)= r(\beta)^{-1} \beta g \beta^{-1} r(\beta)$. Hence $\beta^{-1} r(\beta) r(\alpha)^{-1} \alpha$ commutes with all the elements of $G_1$, thus must be trivial. We get that $ r(\beta \alpha^{-1}) = \beta \alpha^{-1}$ so $\beta \alpha^{-1} \in G'$ (it lies in the image of $r$). Since $G' = H_1* \ldots * H_l$, the factor $G_i$ occurs exactly once as a factor in a conjugate of a Grushko decomposition of one of the $H_j$.

For each $G_i$, let $U_i$ be a space with fundamental group $G_i$. The group $H_j$ can be seen as the fundamental group of a space of the form $X_j = ST(U_{i^j_1}, \ldots, U_{i^{j}_{m_j}}, n_j)$, and $G$ is the fundamental group of a space of the form $Y = ST(U_1, \ldots, U_m, n)$. Note that the embeddings $H_j \to G$ can be represented by immersions $X_j \to Y$ which send each $U_i$ in $X_j$ homeomorphically to the copy of $U_i$ in $Y$. On the other hand, the graph of groups $\Gamma$ gives a way of representing $G$ as the fundamental group of a connected space $X$ obtained by attaching the surface $\Sigma$ to $\sqcup_j X_j$ along its boundary. 

We can thus apply Lemma \ref{GeneralBFLemma}: we get that it is possible to find a decomposition $G'=G'_0*\langle u \rangle$ such that if $b_1, \ldots, b_p$ denote the images in $G'$ of the generators of boundary subgroups of $\Sigma$, we have $b_1=u$ and $b_2, \ldots, b_p$ are contained in conjugates of $G'_0$.

Consider the quotient $q:G' \to \Z/2\Z$ which kills $G'_0$ as well as the square of $u$. Note that the existence of the retraction $r: G \to G'$ guarantees that some product $\alpha = b'_1 \ldots b'_p$ of conjugates of the elements $b_1, \ldots, b_p$ can be written as a product of commutators and squares in $G'$. Hence $q(\alpha)$ should be trivial. On the other hand, we have that $q(u) = q(b_1) \neq 1$ - this is a contradiction.

Thus for some index $i$, say $i=1$, the surface group $S$ is contained in $G_1$. There remains to see that $G_1$ admits a floor structure as claimed. Consider the decomposition $G'= K_1 * \ldots *K_s *F$ induced by the decomposition $G = G_1 * \ldots * G_p$, where $K_1 = G_1 \cap G'$. The restriction of $r$ to $S$ induces an action of the surface group $S$ on the tree corresponding to the free splitting $G'= K_1 * \ldots *K_s *F$ in which the boundary subgroups lie in conjugates of $K_1$. We choose a maximal collection ${\cal C}$ of simple closed curves on $\Sigma$ whose corresponding elements are killed by $r$, and build the map $\rho_{\cal C}$ whose kernel is normally generated by the elements corresponding to ${\cal C}$. Recall that $r$ factors as $r' \circ \rho_{\cal C}$ and that $\rho_{\cal C}(S)$ admits a graph of groups decomposition $\Gamma(\Sigma,{\cal C})$ with trivial edge stabilizers, and vertex stabilizers which are the fundamental groups of surfaces $\Sigma_1, \ldots , \Sigma_q$. Moreover, by maximality of ${\cal C}$, the restriction of $r'$ to each $S_k = \pi_1(\Sigma_k)$ does not kill any element corresponding to a simple closed curve on $\Sigma_k$. In particular, if $\Sigma_k$ contains boundary elements, $r'(S_k)$ lies in a conjugate $\gamma_k K_1 \gamma^{-1}_k$ of $K_1$. We now define a retraction $r_1: G_1 \to K_1$ (respectively $r_1: G_1*\Z \to K_1*\Z$ if $K_1$ is cyclic) by setting $r_1 \mid_S = r'_1 \circ \rho_{\cal C}$ where $r'_1\mid_{S_k}= \Conj(\delta_k \gamma^{-1}_k) \circ r'\mid_{S_k}$ with $\delta_k \in K_1$ (respectively in $K_1 *\Z$) for those $S_k$ which contain boundary elements, and choosing $r'_1\mid_{S_k}: S_k \to K_1$ for those $S_k$ which don't. It is not difficult to see that we can choose these last retractions as well as the elements $\delta_k$ in such a way that $r_1(S)$ is not abelian. This gives the desired floor structure for $G_1$.
\end{proof}
So we see that if some torsion-free hyperbolic group $G$ admits a structure of hyperbolic floor, it essentially comes from a structure of floor of one of the factors $G_i$ of the Grushko decomposition of $G$.

The next lemma will help us to show that in fact $S$ corresponds to a subsurface of some surface group in the JSJ decomposition of this factor $G_i$. 	
\begin{lemma} \label{RetractingSurfComeFromJSJ}
Let $G$ be a freely indecomposable, torsion free hyperbolic group, and $\Gamma$ some splitting of $G$ as a graph of groups with surfaces. Denote by $\Lambda$ the maximal cyclic $JSJ$ decomposition of $G$.

If $\Sigma$ is a surface of $\Gamma$, it is a subsurface of a surface in $\Lambda$. If moreover $\Sigma$ floor-retracts in $\Gamma$, then this subsurface floor-retracts in $\Lambda$.
\end{lemma}

\begin{proof}
Denote by $\Lambda'$ the tree of cylinders of $G$. Since $\Lambda$ is obtained from $\Lambda'$ by the procedure described in Section 2, it is enough to show that $\Sigma$ is a subsurface of $\Lambda'$ which floor-retracts in $\Lambda'$. 

Let $T_{\Gamma},T_{\Lambda'}$ be the $G$-trees corresponding to $\Gamma, \Lambda'$. By Theorem 11.4 in \cite{GuirardelLevittUltimateJSJ} there is a tree $T$ which refines both, such that if $e\in T$ is collapsed in $T_{\Gamma}$ it does not collapse in $T_{\Lambda'}$. So vertex groups in $\Gamma$ are fundamental groups of subgraphs of groups, whose vertices are rigid groups of $\Lambda'$ or subsurfaces of $\Lambda'$. Let $Y$ be the subtree of $T$ which is collapsed to a surface type vertex $v_{\Sigma}$ of $T_{\Gamma}$ : then $Y$ is dual to a set of simple closed curves on $\Sigma$, so any one of its vertex groups corresponds to a subsurface of $\Sigma$, which implies that there is a splitting of $G$ in which it is not elliptic - in particular none of its vertex groups come from rigid vertices of $\Lambda'$. None of the edges of $Y$ collapse in $T_{\Lambda'}$, and since in $T_{\Lambda'}$ there are no adjacent surface-type vertices, $Y$ is a point and corresponds to a subsurface of $T_{\Lambda'}$. Since $\Sigma$ floor-retracts in $\Gamma$ there is a hyperbolic floor structure corresponding to $\Gamma_{\Sigma}$. But the graphs $\Gamma_{\Sigma}$ and $\Lambda'_{\Sigma}$ are identical, so $\Sigma$ floor-retracts in $\Lambda'$. 
\end{proof}

\begin{prop} \label{SubtowersUnderFirstTwoCond} Let $G$ be a freely indecomposable torsion free hyperbolic group, and denote by $\Lambda$ its maximal cyclic JSJ decomposition. Suppose that 
\begin{itemize}
\item none of the surfaces $\Sigma$ of $\Lambda$ admits a proper subsurface $\Sigma'$ which floor-retracts in $\Lambda$;
\item for any set of surfaces which floor retracts in $\Lambda$ with corresponding set of vertices $W$, and for any connected component $\Lambda_0$ of $\Lambda - W$, we have either that $(i)$ $G_0= \pi_1(\Lambda_0)$ is freely indecomposable, and $\Lambda_0$ is the maximal cyclic JSJ decomposition of $G_0$, or $(ii)$ $G_0 \simeq \F(a,b)$ and the edge groups of $\Lambda-\Lambda_0$ adjacent to $\Lambda_0$ are conjugate into $\langle [a,b] \rangle$.
\end{itemize}

If $H$ is a core of $G$ with respect to some tuple, then $G$ admits a floor structure $(G, H * \F, r)$, and there is a set $V_0$ of surface type vertices of $\Lambda$ such that the floor decomposition associated with $(G, H * \F, r)$ is exactly $\Lambda_{V_0}$.  
\end{prop}
In particular the set of surfaces associated to $V_0$ floor-retracts in $\Lambda$, and the tower-retract structure consists of a single floor.

\begin{proof} Suppose the structure of core is given by $G  > G' > \ldots  > H *\F$ with retractions $r,r',...$ and associated decompositions $\Gamma, \Gamma',...$. Since $\Gamma$ is a cyclic graph of groups for $G$ and $G$ is freely indecomposable, by Lemma \ref{RetractingSurfComeFromJSJ} the set of surfaces $\{\Sigma_1, \ldots, \Sigma_p\}$ of 
$\Gamma$ is a set of subsurfaces of surfaces $\Lambda$ which floor-retracts in $\Lambda$, and by the first condition it is in fact a subset of the set of (full) surfaces of $\Lambda$. Thus $G'$ is the free product of fundamental groups of the connected components $\Lambda_1, \ldots, \Lambda_m$ of $\Lambda - \{v_1, \ldots, v_p\}$, where $\{v_1, \ldots, v_p\}$ is a family of surface type vertices which floor-retracts in $\Lambda$.

By the second condition, for each $i$ we have either that $\pi_1(\Lambda_i)$ is freely indecomposable and $JSJ(\pi_1(\Lambda_i)) = \Lambda_i$, or $G_i \simeq \F(a,b)$ and the edge groups of $\Lambda-\Lambda_0$ adjacent to $\Lambda_0$ are conjugate into $\langle [a,b] \rangle$. Suppose without loss of generality that the first condition is satisfied by $i \in \{1, \ldots, l\}$ only.

Now $G'=G_1*\cdots* G_l*\F_{2(m-l)}$ is a Grushko decomposition of $G'$, where $G_i=\pi _1(\Lambda_i)$. Consider the floor structure $(G', G'', r')$ and without loss of generality assume that the associated graph of groups $\Gamma'$ has a unique surface $\Sigma'$ with associated vertex group $S'$. The surface $\Sigma'$ floor-retracts in $\Gamma'$, so Lemma \ref{SurfaceInFreeComponent} implies that up to conjugation we have that $S'$ is contained in one of the factors $G_i$ and there exists a floor structure $(G_i, G'_i, r_i)$ whose associated decomposition $\Gamma_i$ has exactly one surface corresponding to $\Sigma'$. By Lemma \ref{RetractingSurfComeFromJSJ}, we see that $\Sigma'$ must in fact be a subsurface of the JSJ decomposition $\Lambda_i$ of $G_i$ which floor retracts in $\Lambda_i$. Now recall that $\Lambda_i$ is a subgraph of groups of the JSJ decomposition $\Lambda$ of $G$, so $\Sigma'$ floor-retracts in $\Lambda$. By the first condition again, $\Sigma'$ must be a full surface of $\Lambda$, and we may assume that $r'$ restricts to $r_i$ on $G_i$ and to the identity on the other factors of $G'$. By Lemma \ref{SurfacesCanRetractTogether}, we get that $(G,G'',r'\circ r)$ is a floor structure. 

By induction, $G$ has a structure of a single floor retract over $H*\F$ and $H * \F = G'$. By the conditions in the definition of a core, we get that $H = G_1*\cdots* G_l$, which proves the result. 
\end{proof}

\section{Third obstruction to homogeneity: extra symmetries in floor-retracts.} \label{ExtraSymmetries}

There are some more obstructions to homogeneity, which come from the fact that the modular group is not the whole automorphism group (recall Example \ref{SymmetryEx}).

The notion of a homomorphism of graph of groups given in \cite[Definition 2.1]{BassCoveringTheoryGOG} implies that of an isomorphism of a graph of groups. 
\begin{defi} \label{IsoGOGDefi} An \textbf{isomorphism of graphs of groups} $\Phi$ between graphs of groups $\Lambda$ and $\Lambda'$ is given by $(\phi, \{\phi_v\}_{v \in V(\Lambda)},\{\phi_e\}_{e \in E(\Lambda)}, \{\gamma_v\}_{v \in V(\Lambda)},\{\gamma_e\}_{e \in E(\Lambda)})$ where
\begin{itemize}
\item $\phi$  is a graph isomorphism between the underlying graphs;
\item for every $v \in V(\Lambda)$,  $\phi _v:G_v \to G'_{\phi(v)}$ is a group isomorphism and for every $e \in E(\Lambda)$, $\phi _e:G_e \to G'_{\phi(e)}$ is a group isomorphism with $\phi_{\bar{e}}=\phi_{e}$;
\item $\{\gamma_v\}_{v \in V(\Lambda)},\{\gamma_e\}_{e \in E(\Lambda)}$ are sets of elements of $\pi_1(\Lambda')$;
\end{itemize}
The above data are assumed to satisfy that for every $e \in E(\Lambda)$ with $v=o(e)$
\begin{itemize}
\item we have $\gamma_e = \gamma_{v} \delta_e$ for some element $\delta_e$ of $G'_{\phi(v)}$;
\item for any $s \in G_e$, we have $ \gamma_v \phi_{v} (i_e(s)) \gamma^{-1}_{v} = \gamma_e i_{\phi(e)}(\phi_e(s)) \gamma^{-1}_e$. 
 \end{itemize}
\end{defi}

An isomorphism between graphs of groups induces an isomorphism between their fundamental group - this follows from Proposition 2.4 and 2.12 of \cite{BassCoveringTheoryGOG}.
\begin{lemma} \label{IsoGOGGivesISoGroups} Let $\Phi$ be an isomorphism between graphs of groups $\Lambda$ and $\Lambda'$ given by 
$$\Phi=(\phi, \{\phi_v\}_{v \in V(\Lambda)},\{\phi_e\}_{e \in E(\Lambda)}, \{\gamma_v\}_{v \in V(\Lambda)},\{\gamma_e\}_{e \in E(\Lambda)})$$ 
Every choice of a maximal subtree $\Lambda_0$ of $\Lambda$ induces an isomorphism $\phi:\pi_1(\Lambda, \Lambda_0) \to \pi_1(\Lambda', \phi(\Lambda_0))$ such that for any $v \in V(\Lambda)$, and any $s \in G_v$ we have $\phi(s) = \gamma_v \phi_v(s) \gamma^{-1}_v$.
\end{lemma}

In fact we have an equivalent condition for two graphs of groups to be isomorphic.
\begin{lemma} \label{IsoGOGIsoTrees} Suppose $\Lambda, \Lambda'$ are graphs of groups for groups $G,G'$ with associated trees $T, T'$. There exists a graph of group isomorphism between $\Lambda$ and $\Lambda'$ if and only if there exists an isomorphism $\phi: G \to G'$ and a $\phi$-equivariant isomorphism $f: T \to T'$. 
\end{lemma}
This also follows from Proposition 2.4 of \cite{BassCoveringTheoryGOG}

We slightly extend this notion to encompass "disconnected" graphs of groups (which might be obtained by erasing some vertices and edges in a usual graph of groups for example).

\begin{defi} An isomorphism of graphs of groups between collections of graphs of groups $\Lambda_1, \ldots, \Lambda_r$ and $\Lambda'_1, \ldots, \Lambda'_{r}$ consists of a bijection $\sigma: \{1, \ldots , r\} \to \{1, \ldots, r\}$ together with a collection of isomorphisms of graphs of groups $\Phi_i: \Lambda_i \to \Lambda'_{\sigma(i)}$ for $i \in \{1, \ldots, r\}$.
\end{defi} 

To extend isomorphisms between factors of cores of $G$ to automorphisms of $G$, the following result will be central: 
\begin{prop} \label{IsoHypGroupsGivesIsoJSJ} Let $H,K$ be freely indecomposable, torsion free hyperbolic groups. Let $\Lambda_H, \Lambda_K$ be their respective maximal-cyclic JSJ decomposition. If $\phi: H \to K$ is an isomorphism, it induces an isomorphism of graphs of groups between $\Lambda_H$ and $\Lambda_K$.
\end{prop}

\begin{proof} This is a direct consequence of Remark \ref{UniquenessMaxCyclicJSJ}, which says that there is a $\phi$-equivariant isomorphism between the trees $T_H, T_K$ corresponding to $\Lambda_H, \Lambda_K$. By Lemma \ref{IsoGOGIsoTrees}, this means that there is an isomorphism of graphs of groups between $\Lambda_H$ and $\Lambda_K$. 
\end{proof}

We can now give the third obstruction for a group to be homogeneous.
\begin{prop}  Let $G$ be a torsion-free freely indecomposable hyperbolic group which is not the fundamental group of a closed surface of characteristic at least $-2$. Let $\Lambda$ be the maximal-cyclic JSJ decomposition of $G$. 

Let $V_0, W_0$ be sets of socket type vertices of $\Lambda$ whose corresponding surface sets floor-retract in $\Lambda$, and denote by $(\Lambda_0, \ldots, \Lambda_q)$ and $(\Gamma_0, \ldots, \Gamma_t)$ the connected components  of $\Lambda-V_0$ and $\Lambda - W_0$ respectively. Assume that $m \leq q, t$ is such that the only components of $\Lambda-V_0$ (respectively $\Lambda - W_0$) with free fundamental group are $\Lambda_{m+1}, \ldots, \Lambda_q$ and possibly $\Lambda_0$ (respectively $\Gamma_{m+1}, \ldots, \Gamma_t$ and possibly $\Gamma_0$). 

If there is an isomorphism of graphs of groups $\Phi$ between $(\Lambda_0, \ldots, \Lambda_{m})$ and $(\Gamma_0, \ldots, \Gamma_{m})$ such that $\Phi \mid_{\Lambda_0}$ does not extend to an isomorphism of graphs of groups of $\Lambda$, then $G$ is not homogeneous.
\end{prop}

\begin{proof} Let $H_i = \pi_1(\Lambda_i)$ and $K_i = \pi_1(\Gamma_i)$. Note that there are hyperbolic floor structures $(G, H_0* \ldots* H_q, r)$ and $(G, K_0* \ldots* K_t, r')$.

Up to renumbering, there is an isomorphism $\Lambda_i \simeq \Gamma_i$ for each $0 \leq i \leq m$ which, by Proposition \ref{IsoGOGGivesISoGroups}, induces an isomorphism $H_i \to K_i$. Thus we get an isomorphism $\phi: H_0 * \ldots * H_{m} \to K_0 * \ldots * K_m$. 

If $G$ were homogeneous, by Proposition \ref{CriterionForNonHomogeneity} we would get that $\phi \mid_{H_0}$ can be extended to an automorphism of $G$. By Theorem \ref{IsoHypGroupsGivesIsoJSJ}, every such automorphism induces an automorphism of graphs of groups of $\Lambda$, which extends $\Phi \mid_{\Lambda_0}$, so we would get a contradiction.
\end{proof}

\section{Full characterisation of homogeneity}

We can now give a full characterization of homogeneous hyperbolic groups in the freely indecomposable case. 
  \begin{thm} \label{MainFreelyIndec} Let $G$ be a torsion-free freely indecomposable hyperbolic group which is not the fundamental group of a closed surface of characteristic at least $-2$. Let $\Lambda$ be the maximal-cyclic JSJ decomposition of $G$. Then $G$ is homogeneous if and only if
\begin{enumerate}
\item no proper subsurface of a surface of $\Lambda$ floor-retracts;
\item if $V_0 \subseteq V(\Lambda)$ is a set surface type vertices whose corresponding set of surfaces floor-retracts in $\Lambda$, and $\Lambda_1, \ldots, \Lambda_k$ are the connected components of $\Lambda - V_0$, then for each $i$, either $\pi _1(\Lambda_i)$ is freely indecomposable and $\Lambda_i = JSJ(\pi_1(\Lambda_i))$, or $\pi_i(\Lambda_i) = \F(a,b)$ and all edge groups of $\Lambda-\Lambda_i$ adjacent to $\Lambda_i$ are conjugate into $\langle[a,b] \rangle$.
\item if $V_0, W_0$ are sets of surface type vertices whose corresponding surface sets floor-retract in $\Lambda$, and if there is an isomorphism of graphs of groups $\Phi$ between the connected components $(\Lambda_0, \ldots, \Lambda_m)$ and $(\Gamma_0, \ldots, \Gamma_m)$ of  $\Lambda-V_0$ and $\Lambda - W_0$ whose fundamental group is not free (except possibly for $\Lambda_0, \Gamma_0$); then $\Phi \mid_{\Lambda_0}$ extends to an isomorphism of graphs of groups of $\Lambda$.  
\end{enumerate}
\end{thm}

\begin{proof}The  discussion in Sections \ref{RetractingSubsurfacesSection}, \ref{GrowingJSJSection} and \ref{ExtraSymmetries} show these three conditions to be necessary for $G$ to be homogeneous.

Suppose now that all three conditions hold. To see that $G$ is homogeneous, let $\overline h, \overline k$ be tuples in $G$ with $tp_G\left(\overline h\right)=tp_G\left(\overline k\right)$. Let $H,K$ respectively be cores of $\overline h, \overline k$ in $G$. By Proposition \ref{SubtowersUnderFirstTwoCond}, there exist sets of retracting surface type vertices $V_0, W_0$ of $\Lambda, \Gamma$ respectively such that if $\Lambda_0, \ldots , \Lambda_{m}$ and $\Gamma_0, \ldots, \Gamma_{l}$ respectively denote the connected components of $\Lambda - V_0$, $\Lambda-W_0$ respectively, then 
\begin{itemize}
\item $H = \pi_1(\Lambda_0) * \ldots * \pi_1(\Lambda_{m'})$ and $K =  \pi_1(\Gamma_0) * \ldots * \pi_1(\Gamma_{l'})$ where $\Lambda_0, \ldots , \Lambda_{m'}$ and $\Gamma_0, \ldots, \Gamma_{l'}$ are precisely the components of $\Lambda - V_0$ (respectively $\Gamma- W_0$) with non free fundamental groups (except possibly for $\Lambda_0$ and $\Gamma_0$);
\item $\bar{h} \in H_0$ and $\bar{k} \in K_0$.
\end{itemize}
By the second condition, the JSJ of the fundamental group $H_i$ (resp. $K_i$) of $\Lambda_i$ (resp. $\Gamma_i$) is exactly $\Lambda_i$ (resp. $\Gamma_i$) for all $i \leq m$ (respectively $i \leq l$). 

Theorem \ref{EqualTypesImpliesIsomorphismOfCores} implies that there is an isomorphism $\phi:H\to K$ sending $\overline h$ to $\overline k$. By the second condition $H_0 * \ldots * H_{m'}$ (respectively $K_0 * \ldots * K_{l'}$) is a Grushko decomposition for $H$ relative to $\overline h$ (respectively for $K$ relative to $\overline k$), so $m'=l'$ and up to renumbering, we can assume $\phi\mid_{H_i}$ is an isomorphism $H_i \to K_i$. By Theorem \ref{IsoHypGroupsGivesIsoJSJ}, this induces an isomorphism of graphs of groups between $\Lambda_i$ and $\Gamma_i$ for all $i \leq m'$. 

By the third condition, the isomorphism $\phi\mid_{H_0}: \Lambda_0 \to \Gamma_0$ can be extended to an automorphism of graphs of groups of $\Lambda$, which induces an automorphism of $G$ extending $\phi\mid_{H_0}$. 
\end{proof}

We now deal with the general (not necessarily freely indecomposable) case. Let $G$ be a torsion free hyperbolic group, with Grushko decomposition $G = G_1 * \ldots* G_s*\F_r$ where the factors $G_i$ are freely indecomposable and not cyclic.

The following remark is a straightforward consequence of Proposition \ref{SurfaceInFreeComponent}:
\begin{rmk} \label{TowerRetractsInFreeProducts} Let $G$ be a torsion free hyperbolic group, and let $G=G_1 * \ldots * G_t* \F $ be the Grushko decomposition of $G$ relative to some tuple $h$, where $h \in G_1$. 
	
	 If $H$ is a core of $G$ relative to $h$, then $H$ is a free product $H_1* \ldots * H_t*\F$ where $H_i$ is a core of $G_i$ relative to $h$ if $i=1$, and to the trivial element if $i>1$. 
\end{rmk}

We deduce
\begin{cor}  Let $G$ be a torsion free hyperbolic group, with Grushko decomposition $G = G_1 * \ldots* G_s*\F_r$ where the factors $G_i$ are freely indecomposable and not cyclic. If $G$ is homogeneous, then each factor $G_i$ is homogeneous.
\end{cor}

\begin{proof} Suppose without loss of generality that $G_1$ is not homogeneous. Let $u,v$ be tuples in $G_1$ such that $\tp^{G_1}(u)=\tp^{G_1}(v)$ but no automorphism of $G_1$ sends $u$ to $v$. It is enough to show that $\tp^G(u)=\tp^G(v)$ since any automorphism of $G$ sending $u$ to $v$ must send $G_1$ isomorphically onto itself (it is the smallest free factor containing $u$). 
	
	Now by Remark \ref{TowerRetractsInFreeProducts} the core $U$ of $G$ relative to $u$ (respectively $v$) is a free product $U_1* \ldots * U_s*\F_r$ (respectively $V_1* \ldots * V_s*\F_r$) where $U_1, V_1$ are cores of $G_1$ relative to $u, v$, and $U_i, V_i$ are cores of $G_i$ relative to the trivial element if $i>1$. By Theorem \ref{EqualTypesImpliesIsomorphismOfCores},  for each $i$ there is an isomorphism $U_i \to V_i$, sending $u$ to $v$ if $i=1$.  Thus we get an isomorphism $U \to V$ sending $u$ to $v$. This implies that $\tp^{U*\Z}(u) = \tp^{V*\Z}(v)$. By Proposition \ref{TypesOfFactorOfSubtowerArePreserved}, we have $\tp^G(u) = \tp^{U*\Z}(u) = \tp^{V*\Z}(v) = \tp^G(v)$ which terminates the proof.
\end{proof}

To give a general criterion for homogeneity of a torsion free hyperbolic group, we need to slightly generalize the concepts that appear in the criterion we gave in the freely indecomposable case.

\begin{defi} Let $\Lambda =\{\Lambda_1, \ldots, \Lambda_s\}$ be a set of graph of groups (we think of $\Lambda$ as a possibly disconnected graph of groups). Let ${\cal S}$ be a set of surface type vertices of $\Lambda$, that is, ${\cal S} = {\cal S}_1 \cup \ldots \cup {\cal S}_s$ where for each $i$, ${\cal S}_i$ is a set of surfaces of the graph of groups $\Lambda_i$. The set of surfaces associated to ${\cal S}$ is said to floor-retract in $\Lambda$ if each ${\cal S}_i$ floor-retracts in $\Lambda_i$.
\end{defi}

\begin{rmk} Let $G$ be a torsion free hyperbolic group, with $G = G_1 * \ldots* G_s*\F_r$, and for each $i$ let $\Lambda_i$ be a graph of groups with surfaces whose fundamental group is $G_i$. Let $\Lambda =\{\Lambda_1, \ldots, \Lambda_s\}$, and suppose $W$ is a set of surface type vertices of $\Lambda$ whose corresponding surface set floor-retracts. Then there is a floor structure $(G, H, r)$ where $H$ is the free product of the fundamental groups of the connected components of $\Lambda - W$. 
\end{rmk}

We can now state the general criterion for homogeneity
\begin{thm} \label{MainGeneral} Let $G$ be a torsion free hyperbolic group, with $G = G_1 * \ldots* G_s*\F_t$ where the factors $G_i$ are freely indecomposable and not cyclic. Denote by $\Lambda$ the collection of the maximal cyclic JSJ decompositions $\Lambda_1, \ldots, \Lambda_s$ of $G_1, \ldots, G_s$ respectively. Then $G$ is homogeneous if and only if 
	\begin{enumerate}
		\item no proper subsurface of a surface of $\Lambda$ floor-retracts;
		\item if $V_0 \subseteq V(\Lambda)$ is a set surface type vertices whose corresponding set of surfaces floor-retracts in $\Lambda$, and $\Gamma_1, \ldots, \Gamma_k$ are the connected components of $\Lambda - V_0$, then for each $i$, either $\pi _1(\Gamma_i)$ is freely indecomposable and $\Gamma_i = JSJ(\pi_1(\Gamma_i))$, or $\pi_i(\Gamma_i) = \F(a,b)$ and all edge groups of $\Lambda-\Gamma_i$ adjacent to $\Gamma_i$ are conjugate into $\langle[a,b] \rangle$.
		\item if $V_0, W_0$ are sets of surface type vertices whose corresponding surface sets floor-retract in $\Lambda$, and if there is an isomorphism of graphs of groups $\Phi$ between the connected components $(\Gamma_0, \ldots, \Gamma_m)$ and $(\Gamma'_0, \ldots, \Gamma'_m)$ of  $\Lambda-V_0$ and $\Lambda - W_0$ whose fundamental group is not free (except possibly for $\Gamma_0, \Gamma'_0$); then $\Phi \mid_{\Gamma_0}$ extends to an isomorphism of graphs of groups of $\Lambda$.  
	\end{enumerate}
\end{thm}

Before proving this result, we give a description of the structure of cores when the first two conditions are satisfied.
\begin{prop} \label{StructureOfCoresGeneral}  Let $G$ be a torsion free hyperbolic group, with $G = G_1 * \ldots* G_s*\F_t$ where the factors $G_i$ are freely indecomposable and not cyclic. Denote by $\Lambda$ the collection of the maximal cyclic JSJ decompositions $\Lambda_1, \ldots, \Lambda_s$ of $G_1, \ldots, G_s$ respectively. Suppose that the first two conditions of Theorem \ref{MainGeneral} hold, and let $H$ be a core of $G$ relative to some tuple.

Then there exists a set $V$ of surface type vertices of $\Lambda$ which floor retracts in $\Lambda$ such that the corresponding floor decomposition is of the form $(G, H*F, r)$ for some free group $F$.
\end{prop}

\begin{proof} Let $h \in G$, without loss of generality we can assume that the Grushko decomposition for $G$ relative to $h$ is of the form  $G=G_0* G_1 \ldots*G_t*\F$ for some $t\leq s$ (up to renumbering and conjugation of the factors), with $h \in G_0$. Let $(G, G', r)$ be a floor structure relative to $h$, and assume the corresponding floor decomposition has a single surface type vertex with vertex group $S$. By Proposition \ref{SurfaceInFreeComponent}, up to conjugation there is an index $i$ for which $S \leq G_i$ and there exists a floor structure $(G_i, G'_i, r_i)$ such that $r$ restricts to $r_i$ on $G_i$ and to the identity on $\F$ and $G_j$ for $j \neq i$. 
	
The first condition on $\Lambda$ implies that no proper subsurface of $G_i$ floor-retracts in $\Lambda_i$, so by Lemma \ref{RetractingSurfComeFromJSJ} we get that $\Sigma$ is a (full) surface of $\Lambda_i$ which floor-retracts in $\Lambda_i$, and the floor structure corresponding to $(G_i, G'_i, r_i)$ is exactly $\Lambda_{\Sigma}$. Now it is not hard to see that the conditions 1. and 2. also hold for $G'$ and the collection $\Lambda'$ of maximal cyclic JSJ decompositions of the non free factors of the Grushko decompositions of $G'$ (these decompositions are by condition 2. precisely the connected components of $\Lambda - v_{\Sigma}$ whose fundamental group is not free). 

By repeating this arguments for the next floor $(G', G'', r')$ of the structure of core of $H$ in $G$, we get the result.
\end{proof}

We now prove Theorem \ref{MainGeneral}.
 \begin{proof} It is not hard to see that the first two conditions are necessary: indeed if they fail, one of the factors $G_i$ is not homogeneous. To see this, apply Theorem \ref{SubsurfaceRetractNonHomogeneous} and Theorem \ref{GrowingJSJNonHomogeneous} respectively to the factor $G_i$ whose JSJ decomposition $\Lambda_i$ contains the surface $\Sigma$ for the first condition, or the connected component $\Gamma$ of $\Lambda - V_0$ whose JSJ refines for the second.

	If the third condition fails, let $V_0, W_0$ and $\Phi$ witness the failure. We get from $V_0$ and $W_0$ two floor structures $(G, U, r)$ and $(G, U', r')$ for $G$, with decompositions $U=U_0 * \ldots *U_m$ and $U=U'_0*\ldots*U'_m$ induced by the floor decompositions. The existence of $\Phi$ implies that there exists an isomorphism $\phi: U \to U'$ which we may assume sends $U_0$ isomorphically onto $U'_0$. The fact that $\Phi$ witnesses the failure of the third condition means that if $\Gamma_0$ denotes the connected component of $\Lambda-V_0$ whose fundamental group is $U_0$, then $\Phi \mid_{\Gamma_0}$ does not extend to an isomorphism of graphs of groups of $\Lambda$. This implies that $\phi\mid_{U_0}$ does not extend to an isomorphism of $G$. By Proposition \ref{CriterionForNonHomogeneity} the group $G$ is not homogeneous.

To see that these conditions are sufficient, suppose they hold and pick $u,u'$ tuples in $G$ with $\tp^G(u)=\tp^G(u')$. By Proposition \ref{StructureOfCoresGeneral}, if $U,U'$ denote the cores of $G$ relative to $u,u'$ respectively, there exist sets $V_0, W_0$ of surface type vertices of $\Lambda$  which floor-retract with corresponding floor structures $(G, U*F, r)$ and $(G, U'*F', r')$. Then $U= U_0 * \ldots * U_l$ and $U'=U'_0* \ldots*U'_m$ where the maximal cyclic JSJ decompositions of the factors $U_i, U'_i$ are by condition 2. precisely the connected components of $\Lambda - V_0$ and $\Lambda - W_0$ whose fundamental groups are not free relative to $u, u'$. By Theorem \ref{EqualTypesImpliesIsomorphismOfCores}, there is an isomorphism $f: U \to U'$ sending $u$ to $u'$. This implies that there is a graph of groups isomorphism $\Phi$ between the collection of connected components of $\Lambda - V_0$ and $\Lambda - W_0$ whose fundamental groups are not free relative to $u,u'$. By the third condition, $\Phi \mid_{\Gamma_0}$ - where $\Gamma_0$ is the connected component of $\Lambda -V_0$ whose fundamental group contains $u$ - extends to an isomorphism of graph of groups of $\Lambda$. This implies that $f\mid_{U_0}$ extends to an automorphism of $G$. 
 \end{proof}

\providecommand{\bysame}{\leavevmode\hbox to3em{\hrulefill}\thinspace}
\providecommand{\MR}{\relax\ifhmode\unskip\space\fi MR }
\providecommand{\MRhref}[2]{%
	\href{http://www.ams.org/mathscinet-getitem?mr=#1}{#2}
}
\providecommand{\href}[2]{#2}

\end{document}